\documentclass[12pt]{amsart}

\usepackage{amsfonts, amssymb, amsmath, latexsym,  mathtools, amscd,mathrsfs}
\usepackage{hyperref,graphicx}
\usepackage{apxproof}
\usepackage{xcolor}
\usepackage{url}
\usepackage{wrapfig}
\usepackage{subcaption,cleveref}
\usepackage{cite}
\usepackage{float}
\usepackage{tikz}
\usetikzlibrary{positioning}
\usepackage{capt-of}
\usepackage{soul,cancel}
\usepackage{nicematrix}

\DeclareSymbolFont{bbold}{U}{bbold}{m}{n}
\DeclareSymbolFontAlphabet{\mathbbold}{bbold}
\newtheorem{Theorem}{Theorem}[section]

\newtheorem{Lemma}[Theorem]{Lemma}
\newtheorem{Corollary}[Theorem]{Corollary}

\newtheorem{Fact}[Theorem]{Fact}

\theoremstyle{remark}
\newtheorem{Remark}[Theorem]{Remark}
\newtheorem{Example}[Theorem]{Example}
\newtheorem{Definition}[Theorem]{Definition}

\newcommand{\CC}{{\mathbb C}}
\newcommand{\NN}{{\mathbb N}}

\newcommand{\RR}{{\mathbb R}}
\newcommand{\TT}{{\mathbb T}}
\newcommand{\ZZ}{{\mathbb Z}}
\newcommand{\QQ}{{\mathbb Q}}
\renewcommand{\SS}{{\mathbb S}}

\newcommand{\calU}{{\mathcal U}}
\newcommand{\calA}{{\mathcal A}}
\newcommand{\calE}{{\mathcal E}}

\newcommand{\calV}{{\mathcal V}}

\newcommand{\Span}{{\text{Div}}}

\newcommand{\calN}{{\mathcal N}}

\newcommand{\newt}[1]{\calN(#1)}

\newcommand{\defcolor}[1]{{\color{blue}#1}}
\newcommand{\demph}[1]{\defcolor{{\sl #1}}}

\definecolor{TAMU}{RGB}{140,0,0}
\definecolor{myblue}{RGB}{0,0,198}
\definecolor{myred}{RGB}{182,0,0}
\newcommand{\bfz}{{\bf  0}}
\title[Irreducibility of the Dispersion Polynomial for Periodic Graphs]{Irreducibility of the Dispersion Polynomial \\ for Periodic Graphs}

\author{M.~Faust}
\address{Matthew Faust, Department of Mathematics,
         Texas A\&M University, College Station, Texas 77843,  USA}
\email{mfaust@tamu.edu}
\urladdr{https://mattfaust.github.io/}

\author{J.~Lopez Garcia}
\address{Jordy Lopez Garcia, Department of Mathematics,
         Texas A\&M University, College Station, Texas 77843,  USA}
\email{jordy.lopez@tamu.edu}
\urladdr{https://jordylopez27.github.io/}
\thanks{Research supported in part by NSF DMS-2201005, DMS-2000345, DMS-2052572, and DMS-2246031.}
\subjclass[2010]{14M25,47A75,52B20.}
\keywords{Bloch Variety, Fermi Variety, Dispersion Polynomial, Irreducibility}
\begin{document}
                               
\begin{abstract}
 We use methods from algebra and discrete geometry to study the irreducibility of the dispersion polynomial of a discrete periodic operator associated to a periodic graph after changing the period lattice. We provide numerous applications of these results to discrete periodic operators associated to families of graphs which include dense periodic graphs, and a family containing the hexagonal and diamond lattices. 
\end{abstract}

\maketitle 

\section*{Introduction}

Let $\Gamma$ be a $\ZZ^{d}$-periodic graph, and let \demph{$\ell^2(\Gamma)$} be the Hilbert space of square-summable functions on the vertices of $\Gamma$. A \demph{discrete periodic operator} $L$--a weighted graph Laplacian with a periodic potential--acting on $\ell^{2}(\Gamma)$ is self-adjoint and therefore has real spectrum $\sigma(L)$. After applying the Floquet transform to $\ell^{2}(\Gamma)$, the operator may be represented by a matrix of \demph{Laurent polynomial} entries. The characteristic polynomial of this matrix is called the \demph{dispersion polynomial} (or the Bloch polynomial), and its zero-set is an algebraic hypersurface in $(\SS^{1})^{d}{\times}\RR$, where $\SS^{1}$ is the complex unit circle. This variety is known as the \demph{Bloch} \demph{variety}. We may recover the spectrum of this operator by projecting the Bloch variety onto $\RR$. The level set of the Bloch variety at the eigenvalue $\lambda_0 
$ $\in \sigma(L)$ is the \demph{Fermi} \demph{variety} at energy level $\lambda_0$.

Studying irreducibility of these varieties may seem esoteric, but irreducibility has important consequences. In~\cite{KV, Kuchment1998, KV22} it was shown that irreducibility of the Fermi varieties implies the absence of embedded eigenvalues (see~\cite{Overview, Embed,Liu2022} for related applications). Moreover, it was proven in~\cite{LiuQE} that irreducibility of the Bloch variety implies quantum ergodicity. In~\cite{Battig, FLM, LiShip, GKT, Liu2022}, these varieties were shown to be irreducible for operators defined on a class of finite-range graphs, such as the square lattice and certain planar periodic graphs. Reducibility of the Bloch variety has also been studied in~\cite{LeeLiShip ,Shipman2, Shipman,sabri2023flat}. For a more detailed history, see~\cite{Overview,kuchment2023analytic, TopFermiLiu}.

In this work we study the (ir)reducibility of the Bloch and Fermi varieties upon a change of the period lattice. This line of work dates back to the 1980s, with a focus on the discrete periodic Schr\"odinger operator. For this operator, irreducibility of the Bloch variety was proven in~\cite{BattigThesis} for $d=2$. Moreover, it was shown in ~\cite{GKT} that for $d=2$ all but finitely many of the Fermi varieties are irreducible, and in ~\cite{Battig} it was shown that every Fermi variety is irreducible for $d=3$. In~\cite{Liu2022}, this result was extended to higher dimensions. Most recently, irreducibility of Bloch varieties for a large class of graphs, which includes the triangular lattice and Harper lattice, was proven in~\cite{FLM}. In ~\cite{fillman2023algebraic}, irreducibility of all but finitely many Fermi varieties for a large class of graphs was demonstrated, including the Lieb lattice.

As irreducibility of the dispersion polynomial implies irreducibility of the Bloch variety, we examine the effect that changing the period lattice has on the irreducibility of this polynomial. We apply our observations to discuss the irreducibility of the Bloch varieties associated to various families of periodic graphs, including the diamond lattice (Example~\ref{GrapheneBloch}), and dense periodic graphs (these are defined before Example~\ref{ex:DenseBloch} and were first introduced in~\cite{FS}). We also discuss the irreducibility of Fermi varieties associated to dense periodic graphs (Example~\ref{DFermi}).

Fixing a $\ZZ^d$-periodic potential, we study the reducibility of the dispersion polynomial $D(z,\lambda)$ after replacing the period lattice with the sublattice $Q \ZZ =  \bigoplus_{i=1}^{d}q_{i}\ZZ$, where $Q \in \NN^d$. We show the irreducibility of the corresponding dispersion polynomial $D_{Q}(z,\lambda)$ depends on the irreducibility of $D(z,\lambda)$ and a relationship between the support of $D(z,\lambda)$ and the tuple $Q$ (Theorem~\ref{TH:1} and Lemma~\ref{lem:coprime}).

We use discrete geometry to study when irreducibility is preserved for a potential that is periodic with respect to the sublattice $Q\ZZ$. Roughly, for a $Q\ZZ$-periodic potential, we show that if enough of the facial polynomials of $D_{Q}(z,\lambda)$ are also facial polynomials of a $\ZZ^d$-periodic potential and the corresponding facial polynomials of $D(z,\lambda)$ are irreducible, then the polynomial factors ``only homothetically'' (Corollary~\ref{cor:ZdPeriodic}). If this condition is met, then $D_{Q}(z,\lambda)$ is irreducible if it has an irreducible facial polynomial (Corollary~\ref{cor:irred}). The theory of these \demph{only homothetically reducible} polynomials, which maybe of independent interest, is explored in Section~\ref{sec:3-Bragg}, and applied in Section~\ref{sec:4-applications} to study the irreducibility of $D_Q(z,\lambda)$ associated to various periodic graphs for $Q\ZZ$-periodic potentials.

There is overlap between our methods and those in ~\cite{FLM}, which were inspired by the work of ~\cite{Liu2022}. In particular, in all of these works an interplay between the structure of the dispersion polynomial and the irreducibility of their facial polynomials is exploited in order to discuss the irreducibility of the dispersion polynomial. We view our main contribution as translating the heart of these arguments to the language of discrete geometry. This allows us to utilize the theory of indecomposability for integral polytopes to the study of irreducibility of the Bloch and Fermi varieties. This new lens enables us to apply our results to study the irreducibility of a larger class of dispersion polynomials; for example, those that arise from many-vertex models like the hexagonal lattice, as opposed to the single-vertex models (see 1-vertex graphs in Section~\ref{sec:4-applications}), such as the square lattice, that are the subjects of~\cite{FLM}. Remarks~\ref{remark:past} and \ref{rm:compare} will further compare and contrast our results with those of~\cite{FLM}.

Section~\ref{sec:1-prelims} introduces the necessary background on discrete periodic operators, algebra, and discrete geometry.  Section~\ref{sec:3-Bragg} uses classical results on the decomposability of polytopes to obtain analogous results for a class of Laurent polynomials.  For a $\ZZ^d$-periodic potential, Section~\ref{sec:2-structure} discusses sufficient conditions for when the irreducibility of the dispersion polynomial is preserved after changing the period lattice, which when combined with Section~\ref{sec:3-Bragg}, will enable us to discuss irreducibility of the dispersion polynomial for more general potentials. Section~\ref{sec:4-applications} provides various applications of these results.

\section{Preliminaries}\label{sec:1-prelims}

\subsection{Periodic Graphs and Discrete Periodic Operators}\label{subsec:periodicg}
For more, see \cite{BerkKuch,KuchBook,TopCry}. 
 Let $G$ be a free abelian group. A locally finite simple (undirected, and no multiple edges or loops) graph \demph{$\Gamma$} $\vcentcolon= (\calV(\Gamma),\calE(\Gamma))$ is \demph{$G$-periodic} if it is equipped with a (free) co-finite action of $G$, called the (\demph{abstract}) \demph{period lattice}. A \demph{fundamental domain} is a set \demph{$W$} of representatives of $G$-orbits of $\calV(\Gamma)$. Suppose $G$ has rank $d$ and assume $G\cong \ZZ^{d}$. We write the action of $G$ on $\Gamma$ as addition: for $a\in G$, $v\in \calV(\Gamma)$ implies $v+a\in \calV(\Gamma)$, and $(u,v)\in \calE(\Gamma)$ implies $(u,v)+a = (u+a,v+a)\in \calE(\Gamma)$. Figure~\ref{fig:periodic-graphs} illustrates two $\ZZ^{2}$-periodic graphs.  We call the collection of $a \in \ZZ^d$ such that $W$ shares an edge with $a+W$ the \demph{support} of $W$, and we denote it by \demph{$\calA(W)$}. 

\begin{figure}[htb]
     \centering
     \begin{subfigure}{0.32\textwidth}
         \centering
         \includegraphics[width=\textwidth]{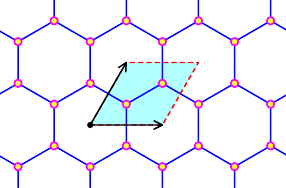}
         \label{fig:graphene}
     \end{subfigure}
     \hspace{0.5in}
      \begin{subfigure}{0.3\textwidth}
         \centering \includegraphics[width=\textwidth]{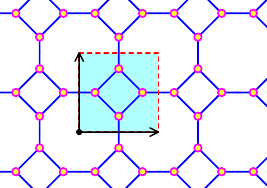}
         \label{fig:K4}
      \end{subfigure}
\caption{The left figure is the hexagonal lattice. The right figure is an abelian cover of $K_{4}$.}
\label{fig:periodic-graphs}
\end{figure}

Let $\Gamma$ be a $\ZZ^d$-periodic graph with fundamental domain $W$. For $Q\vcentcolon = (q_1,\dots, q_d) \in \NN^d$, let \demph{$Q\ZZ$} denote the finite-index subgroup $\bigoplus_{i=1}^{d}q_{i}\ZZ$ of $\ZZ^d$. As $\Gamma$ is also $Q\ZZ$-periodic, the fundamental domain $W$ induces a choice of fundamental domain for $\Gamma$ as a $Q\ZZ$-periodic graph, namely \demph{$W_Q$} $= \bigcup_{0 \le k_i < q_i} (W+k)$, where $k = (k_1,\dots, k_d) \in \ZZ^d$ and $W + k= \{u + k \mid u \in W \}$. We call $W_Q$ the \demph{$Q$-expansion} of $W$. Figure~\ref{fig:submodule-action} gives the $(3,2)$-expansion of the fundamental domain of the hexagonal lattice shown in Figure~\ref{fig:periodic-graphs}. 
\begin{figure}[ht]
     \centering
         \centering \includegraphics[scale=1.2]{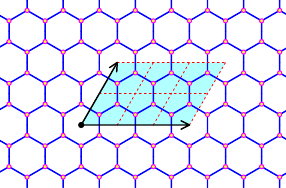}
\caption{The $(3,2)\ZZ$-periodic hexagonal lattice.}
\label{fig:submodule-action}
\end{figure}
A (\demph{$G$-periodic}) \demph{potential} is a function $V\vcentcolon\calV(\Gamma)\to \RR$ such that $V(v{+}a)=V(v)$ for all $a \in G$ and $v \in \calV(\Gamma)$, and a (\demph{$G$-periodic}) \demph{edge labeling} is a function $E\vcentcolon \calE(\Gamma) \to \RR$ such that $E_{(u{+}a,v{+}a)}=E_{(u,v)}$ for all $a \in G$ and $(u,v)\in \calE(\Gamma)$. Together, the pair \demph{$c$} $\vcentcolon=$ $(V,E)$ defines a \demph{labeling} of $\Gamma$. Given a labeling $c =(V,E)$, the \demph{discrete periodic operator $L_{c}$} acting on a function $f$ on $\calV(\Gamma)$ is defined by its value at $u\in \calV(\Gamma) \vcentcolon$
\[
(L_{c}f)(u)\vcentcolon =V(u)f(u)+\sum_{(u,v){\in}\calE(\Gamma)}E_{(u,v)}(f(u)-f(v)).
\]
The second summand is called a \demph{weighted graph Laplacian}, and it is denoted by \demph{$\Delta_{E}$} (and thus $L_{c} = V + \Delta_{E}$).

Discrete periodic operators are bounded and self-adjoint on the Hilbert space of square summable functions on $\calV(\Gamma)$, denoted by \demph{$\ell^{2}(\Gamma)$}.
When $E$ is the constant function $E=1$, $L_{c}$ is a \demph{discrete periodic Schr\"odinger operator}. When $V$ and $E$ are the constant functions $V=0$ and $E=1$, $L_{c}$ is the \demph{graph Laplacian}.

\medskip

\subsection{Floquet Theory}~\label{sec:floquettheory}
For more, see \cite{BerkKuch,FloquetTheory, Overview}. 
Let \demph{$\CC^{\times}$} be the multiplicative group of nonzero complex numbers. The \demph{torus} is the maximal compact subgroup of $\CC^{\times}$ defined by \demph{$\TT$} $\vcentcolon=$ $\{z\in \CC^{\times} \mid \lvert z \rvert = 1\}$. Vectors $z\vcentcolon =\left(z_1,\dots, z_d\right)\in \TT^d$ are called \demph{Floquet multipliers}. The torus $\TT^d$ is the group of unitary characters of $\ZZ^{d}$. 
For $z\in \TT^d$ and $a=\left(a_1,\dots, a_d\right)\in \ZZ^d$, the character value is \demph{$z^{a}\vcentcolon =z_{1}^{a_{1}}\cdots z_{d}^{a_{d}}$}. 

Let $\Gamma$ be a $\ZZ^d$-periodic graph with a fundamental domain $W$. We wish to study the spectrum of a discrete periodic operator $L$ ($=L_c$ for some labeling $c$) acting on $\ell^{2}(\Gamma)$.  For $z\in \TT^{d}$, the \demph{Floquet transform} \demph{$\mathscr{F}$} of a function $g$ on $\calV(\Gamma)$ is given by 
\[ 
g(v) \mapsto \hat{g}(z,v) = \sum_{a \in \ZZ^d} g(v+a)z^{-a}.
\]
Since  $\hat{g}(z,u{+}a) = z^a \hat{g}(z,u)$ for $a\in \ZZ^{d}$, $\hat{g}$ is uniquely determined by its values on $W$. One can see that $\hat{g}(z,u)$ is just the Fourier transform of $g$ restricted to vertices in the $\ZZ^{d}$-orbit of $u$ for $u\in W$. If $g \in \ell^{2}(\Gamma)$, then 
\[\sum_{u \in W} \int_{\TT^d} \lvert \hat{g}(z,u) \rvert^2 dz  < \infty. 
\] 
That is, $\hat{g}(z,u)$ is in \demph{$L^{2}(\TT^{d})$}, the Hilbert space of square-integrable functions on $\TT^d$, for each $u \in W$.

Suppose that $W = \{u_{1},\dots,u_{m}\}$. Given a function $g \in \ell^{2}(\Gamma)$, we can view $\hat{g}(z,\cdot \ )$ as  vector-valued function $(\hat{g}(z,u_1), \hat{g}(z,u_2), \dots, \hat{g}(z,u_m))^T$ in the Hilbert space 
\[
L^2(\TT^d)^{W} = \{ \hat{g}(z,\cdot \ ) \mid \hat{g}(z,u_i) \text{ is in } L^2(\TT^d) \text{ for each } i \in [n] \}.
\]

Like the Fourier transform, the Floquet transform $\mathscr{F}\vcentcolon \ell^{2}(\Gamma) \to L^2(\TT^d)^W$ is a unitary operator, so $L$ and $\mathscr{F} L \mathscr{F}^*$ are unitarily equivalent (where $\mathscr{F}^*$ denotes the inverse Floquet transform).

Define \demph{$L(z)$} to be the $\lvert W\rvert{\times\lvert W\rvert}$ matrix such that, for each $u,v \in W$, 
\begin{equation}~\label{eq:Floq}
L(z)_{u,v} = \delta_{u,v}\left(V(u) + \sum_{(u,\omega) \in \calE(\Gamma)} E((u,\omega)) \right)-\sum_{(u,v{+}a)\in \calE(\Gamma)}E((u,v{+}a))z^{a},
\end{equation} 
where $\delta$ is the Kronecker delta function ($\delta_{u,v} = 1$ when $u=v$ and $\delta_{u,v}=0$ otherwise).  

We will now show that, for any $\hat{g}(z,\cdot \ ) \in L^2(\TT^d)^{W}$,
\begin{equation}~\label{eq:decomposition}(\mathscr{F} L \mathscr{F}^* \hat{g})(z,\cdot \ ) = L(z) \hat{g}(z,\cdot \ ).
\end{equation}
For each $z \in \TT^d$ and $u \in W$, 
\[
(\mathscr{F} L \mathscr{F}^*\hat{g})(z,u) = V(u) \hat{g}(z,u) + \sum_{(u,v{+}a) \in \calE(\Gamma), v \in W} E((u,v{+}a))(\hat{g}(z,u) - z^a \hat{g}(z,v)).
\] 
Notice that $L(z)$ is defined such that for each $z \in \TT^d$ and any $\hat{g}(z,\cdot \ ) \in L^2(\TT^d)^{W}$,
\[
L(z)_{u,v} =  [\hat{g}(z,v)] (\mathscr{F} L \mathscr{F}^* (\hat{g})(z,u)), 
\]
where $[\hat{g}(z,v)] f$ denotes extracting the coefficient of $\hat{g}(z,v)$ in $f$. In this way, 
\[
 (\mathscr{F} L \mathscr{F}^*\hat{g})(z,u) = \sum_{v \in W} L(z)_{u,v} \hat{g}(z,v).
\]
It follows that for each $z \in \TT^d$ and any $\hat{g}(z,\cdot \ ) \in L^2(\TT^d)^{W}$, $(\mathscr{F} L \mathscr{F}^*)\hat{g}(z,\cdot \ ) = L(z) \hat{g}(z,\cdot \ )$. 
That is, $\mathscr{F} L \mathscr{F}^*$ acts as multiplication by the $\lvert W\rvert{\times}\lvert W\rvert$ matrix $L(z)$ on the vector $\hat{g}(z,\cdot \ ) \in L^2(\TT^d)^{W}$ for each fixed $z \in \TT^d$. 
We obtain the following fact. 
\begin{Fact} 
Let $I$ be the identity map on $L^{2}(\TT^{d})$ and $\lambda\in \CC$. The map $\mathscr{F} L \mathscr{F}^* - \lambda I$ $\vcentcolon L^{2}(\TT^{d})^{W} \to L^{2}(\TT^{d})^{W}$ is a bijection if and only if for each $z\in \TT^{d}$, $L(z) - \lambda I$ $\vcentcolon L^{2}(\TT^{d})^{W}\to L^{2}(\TT^{d})^{W}$ is a bijection. It follows that, the spectrum of $L$ is given by the union of the eigenvalues of $L(z)$ for each $z \in \TT^d$. 
\end{Fact}

When $z$ is viewed as an indeterminate, it is easy to see that $L(z)$ is just a $|W|{\times}|W|$ matrix with Laurent polynomial entries in $\CC[z^{\pm}]$. Given a discrete periodic operator $L$, we will often refer to $L(z)$ as the \demph{Floquet matrix} of $L$.

\medskip
\begin{Remark}~\label{Rm:Basis}
A  \demph{Floquet function} $f \vcentcolon \calV(\Gamma) \to \CC$ with Floquet multiplier $z \in \TT^d$ is \demph{quasi-periodic} with respect to the $\ZZ^d$-action $\vcentcolon$ for $a \in \ZZ^d$ and $u \in \calV(\Gamma)$, $f(u{+}a) = z^a f(u)$. Such functions are also sometimes called Bloch functions, or functions satisfying ``Floquet-boundary'' conditions. 

We may also obtain the spectrum of $L$ on $\ell^{2}(\Gamma)$ by studying the spectrum of $L$ acting on the spaces of Floquet functions $f : \calV(\Gamma) \to \CC$ with Floquet multiplier $z$ for each $z \in \TT^d$. It is easy to see that such quasi-periodic functions on $\calV(\Gamma)$ make up the generalized eigenfunctions of $L$, see~\cite{DamanikFillman}.
\hfill $\diamond$
\end{Remark}

\begin{Example}\label{ex:floquet_graphene} 
Let $\Gamma$ be the hexagonal lattice with edges labeled as in Figure~\ref{fig:graphene_labeled}, let $W = \{u,v\}$, and let $x,y$ be Floquet multipliers. After the Floquet transform, the discrete periodic operator becomes
\[
\begin{split}(L_c\hat{f})(u) & = V(u)\hat{f}(u)+ \alpha\left( \hat{f}(u) - \hat{f}(v) \right) + \beta\left( \hat{f}(u) - x^{-1}\hat{f}(v) \right) + \gamma\left( \hat{f}(u) - y^{-1}\hat{f}(v) \right),\\
(L_c\hat{f})(v) & = V(v)\hat{f}(v)+ \alpha\left( \hat{f}(v) - \hat{f}(u) \right) + \beta\left( \hat{f}(v) - x\hat{f}(u) \right) + \gamma\left( \hat{f}(v) - y\hat{f}(u) \right).\\
\end{split}
\] 
\noindent Collecting the coefficients from $\hat{f}(u)$ and $\hat{f}(v)$, the Floquet matrix of $L_c$ is 
\[
L_c(x,y) = \begin{pmatrix}V(u){+}\alpha{+}\beta{+}\gamma & -\alpha{-}\beta x^{-1}{-}\gamma y^{-1}\\ -\alpha{-}\beta x{-}\gamma y & V(v){+}\alpha{+}\beta{+}\gamma\end{pmatrix}.
\]

\vspace{-0.8cm}\hfill $\diamond$
\end{Example}
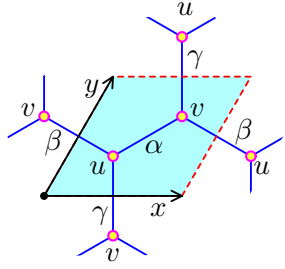
\begin{figure}[h]
\centering
   \begin{picture}(100,90)
     \put(0,0){\includegraphics{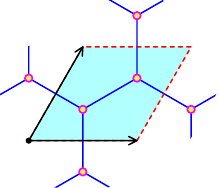}}
     \put(31,31){\small$u$}      \put(93,31){\small$u$}     \put(64,91){\small$u$}
     \put(69,53){\small$v$}      \put( 5,53){\small$v$}     \put(37, -2){\small$v$}
     \put(54,15){\small$x$}
     \put(29,62){\small$y$}
     \put(52,39){\small$\alpha$}
     \put(32,14){\small$\gamma$}   \put(68,73){\small$\gamma$}
     \put(14,40){\small$\beta$}    \put(86,44){\small$\beta$}
   \end{picture}
\caption{The hexagonal lattice with edge labels $(\alpha,\beta,\gamma)$.} 
\label{fig:graphene_labeled}
\end{figure}
The \demph{dispersion polynomial} \demph{$D_{c}(z,\lambda)$} $\vcentcolon =\det(L_{c}(z)-\lambda I)$ is the characteristic polynomial of $L_{c}(z)$, and the \demph{Bloch variety $B_{c}(\Gamma)$} is the zero-set of $D_c(z,\lambda)$, \[ B_{c}(\Gamma) = \{(z,\lambda)\in \mathbb{T}^{d}{\times}\RR \mid D_c(z,\lambda)=0\}. \]  The projection $\pi(z,\lambda) = \lambda$ of the Bloch variety recovers the spectrum of $L_c(z)$, as depicted in Figure~\ref{fig:blochvar}.  We will often omit the labeling, simply writing $L(z)$, $D(z,\lambda)$, and $B(\Gamma)$. The \demph{Fermi variety} at energy $\lambda_{0}$, is the level-set $F_{\lambda_0}\vcentcolon =\{ z \in \mathbb{T}^{d} \mid D(z,\lambda_{0})=0\}$.
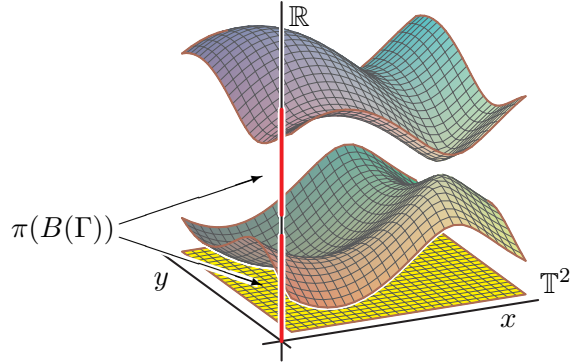
\begin{figure}[ht]
    \centering
      \begin{picture}(178,138)(-39,0)
     \put(0,0){\includegraphics[height=140.4pt]{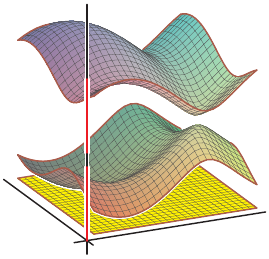}}
     \put(130,15){\small$x$} \put(-1,31){\small$y$}
     \put(50,129){\small$\RR$}    \put(145,28){\small$\TT^2$}
     \put(-55,49){\small$\pi(B(\Gamma))$}
     \thicklines
       \put(-14,50){{\color{white}\line(3,-1){52}}}
       \put(-14,49.5){{\color{white}\line(3,-1){52}}}
       \put(-14,49){{\color{white}\line(3,-1){52}}}
       \put(-14,48.5){{\color{white}\line(3,-1){52}}}
       \put(-14,48){{\color{white}\line(3,-1){52}}}
     \thinlines

       \put(-14,54){\vector(3,1){54}}    \put(-14,49){\vector(3,-1){54}}
   \end{picture}
    \caption{A Bloch variety of the hexagonal lattice,
    $(\alpha,\beta,\gamma) = (6,3,2)$.}
    \label{fig:blochvar}
\end{figure}

 It is natural to allow $V$ and $E$ to take complex values and variables $z \in (\CC^\times)^d$. In this way, $D(z,\lambda)$ becomes a polynomial in $\CC[z^{\pm}, \lambda]$, and the Bloch variety becomes an algebraic hypersurface in $(\CC^\times)^d{\times}\CC$.  When allowing $V$ and $E$ to take complex values, $L(z)$ may no longer be self-adjoint for $z \in \TT^d$, but it still satisfies $L(z) = L(z^{-1})^T$.

\subsection{Changing the Period Lattice of the Potential} \label{coveringsubsec}
Let $\Gamma$ be a $\ZZ^d$-periodic graph with fundamental domain $W$, where \demph{$m$} $\vcentcolon=$ $\lvert W \rvert$, and let $L$ be a discrete periodic operator acting on $\ell^{2}(\Gamma)$. Fix $Q$ $\vcentcolon = (q_{1},\dots,q_{d})$ $\in 
\NN^d$ and let \demph{$\lvert Q \rvert$} $\vcentcolon =\prod_{i=1}^{d}q_{i}$. We wish to study the dispersion polynomial $D(z,\lambda)$ of $L$ with labeling $(V_Q,E)$, where $E \vcentcolon \calE(\Gamma) \to \CC$ is a $\ZZ^d$-periodic edge labeling and $V_Q\vcentcolon \calV(\Gamma) \to \CC$ is a $Q\ZZ^d$-periodic potential, rather than a $\ZZ^d$-periodic potential. We will denote operator $L$ with labeling $(V_Q,E)$ by \demph{$L_{Q}$}.

As $Q\ZZ$ is a free full rank subgroup of $\ZZ^{d}$, the $\ZZ^{d}$-periodic graph $\Gamma$ is also $Q\ZZ$-periodic. Thus, as discussed in Section~\ref{subsec:periodicg}, $W_Q$ is a fundamental domain with respect to the action of $Q\ZZ$. The Floquet matrix of $L_Q$ with respect to the $Q\ZZ$-periodic graph $\Gamma$ with fundamental domain $W_Q$ is denoted by \demph{$L_Q(z)$}. Since $\lvert W_{Q}\rvert =$ $|Q|m$, $L_Q(z)$ is a $\lvert Q \rvert m \times \lvert Q \rvert m$ matrix of Laurent polynomial entries.

We now discuss an alternative representative of $L_Q(z)$ that comes from a change of basis and after a change of variables. For each $z \in \TT^d$, $L_Q(z)$ acts on the space of Floquet functions on $\calV(\Gamma)$ with Floquet multiplier $z$ (see Remark~\ref{Rm:Basis}). Consider the surjective group homomorphism
\begin{align}~\label{eq:covering}
\begin{split}
    \phi : \quad (\CC^{\times})^{d}\quad &\longrightarrow \quad (\CC^{\times})^{d}\\ (z_{1},\dots,z_{d}) &\longmapsto (z_{1}^{q_{1}},\dots,z_{d}^{q_{d}}), 
\end{split}
\end{align}
with \demph{kernel group ~$\calU_Q$} $\vcentcolon=$ $\prod_{i=1}^{d}\calU_{q_{i}}$, where ~$\calU_{q_{i}}$ is the multiplicative group of $q_{i}$th roots of unity. For each $\rho\in \calU_{Q}$ and $u\in W$, consider the function $e_{\rho,u}\vcentcolon \TT^{d}{\times}\ZZ^{d}\to \CC$ with value $1$ at $u\in W\vcentcolon$
\begin{equation*}
    e_{\rho,u}(z,a) = (\rho z)^{a} \vcentcolon = \prod_{i=1}^{d}(\rho_{i}z_{i})^{a_{i}}.
\end{equation*}
For each $\rho\in \calU_{Q}$, we obtain a well-defined function
\begin{equation}~\label{eq:basis}
\begin{split}
    e_{\rho}\vcentcolon \TT^{d}{\times}\calV(\Gamma) &\longrightarrow \CC\\
    \quad (z,v) & \longmapsto e_{\rho,u}(z,a),
\end{split}
\end{equation}
where $v = u{+}a$ for $a\in \ZZ^{d}$ and $u\in W$. Since $e_{\rho}(z,b{+}u) = (\rho z)^{b}e_{\rho}(z,u)$ for $b\in \ZZ^{d}$, $e_{\rho}$ is uniquely determined by its values on $W$.

Writing \demph{$Q_{i}$} $\vcentcolon = q_{i}\varepsilon_{i}$, where $\varepsilon_{i}$ is the $i$th standard basis vector of $\RR^{d}$, it follows that for $u\in W$, $e_{\rho}(z,Q_{i}{+}u{+}a) = z_{i}^{q_{i}}e_{\rho}(z,u{+}a)$. The following lemma is well-known. 

\begin{Lemma}{\cite[Lemma 2.2]{GKT}}
     The set $\{e_{\rho}(z,\cdot \ ) \mid \rho\in \calU_{Q}\}$ forms a basis of the vector space of functions $\psi$ on $\calV(\Gamma)$ satisfying $\psi(Q_{i}{+}u{+}a) = z_{i}^{q_{i}}\psi(u{+}a)$ for $u\in W$ and $i\in [d]$.
\end{Lemma}

That is, the functions $e_{\rho}$ form a basis for Floquet functions with Floquet multiplier $z^{Q}$ with respect to the $Q\ZZ$-action. For each $z$, these are exactly the generalized eigenfunctions of $L_Q$ after the cover map~\eqref{eq:covering}, and thus we may obtain a new matrix representation for $L_Q(z_1^{q_1},\dots, z_d^{q_d})$ in the basis given by the functions of ~\eqref{eq:basis}. For each $\rho\in \calU_{Q}$, the weighted graph Laplacian in the basis~\eqref{eq:basis} is 
\[
(\Delta_E e_{\rho})(z,u) = \sum_{(u,v+a) \in \calE(\Gamma)} E_{(u,v+a)} (e_{\rho} (z,u) - (\rho z)^a e_{\rho}(z,v)),
\]
where $u,v \in W$. The action of $\Delta_E$ with respect to the basis $\{e_{\rho}( z,u) \mid \rho \in \calU_Q , u \in W\}$ is given by multiplication by a diagonal matrix \demph{$\hat{\Delta}_E(z)$}, whose entries are matrices indexed by $\calU_{Q}\times\calU_{Q}$, where each diagonal entry $\hat{\Delta}_E(z)_{\rho, \rho}$ is the $m \times m$ matrix indexed by $W \times W$ given by 
\begin{equation*}
\Delta_E(\rho z)_{u,v} = \delta_{u,v}\left(\sum_{(u,w)\in\calE(\Gamma)}E_{(u,w)}\right) - \sum_{(u,v+a)\in\calE(\Gamma)}E_{(u,v+a)}(\rho z)^{a}.
\end{equation*}
In other words, $\hat{\Delta}_{E}(z)$ is the block-diagonal matrix given by $\lvert Q\rvert{\times}\lvert Q\rvert$ blocks of $m{\times}m$ matrices $\vcentcolon$
\begin{equation*}
    \left(\hat{\Delta}_{E}(z)\right)_{\rho,\rho'}\vcentcolon = \delta_{\rho,\rho'}\cdot \Delta_{E}(\rho z),
\end{equation*}
where the submatrix $\Delta_{E}(\rho z)$ represents $\Delta_{E}$ with Floquet multiplier $\rho z$ (with respect to the $\ZZ^{d}$-action). Explicitly,
\begin{equation*}
\hat{\Delta}_{E}(z) = \quad
\begin{pNiceArray}{c|c}[first-col,first-row]
&         \ \ (u,\rho)\ \ (v,\rho) \ \ \cdots & \ \ (u,\rho') \ \ (v,\rho') \ \ \cdots\\
(u,\rho) & & \\
&\Delta_{E}(\rho z)& \hspace{-0.2in}0\\
(v,\rho) & & \\
\vdots   & & \\
\hline
(u,\rho')& & \\
& \hspace{-0.2in}0 & \Delta_{E}(\rho' z)\\
(v,\rho')& & \\
\vdots   & & \\
\end{pNiceArray}.
\end{equation*}

In order to discuss the potential $V$ in this new basis, we will take a discrete Fourier transform. For each $\mu\in \calU_{Q}$ and $u{+}a\in W_{Q}$, the discrete Fourier transform of the potential $V$ is
\[
(Ve_\mu)(z,v+a) = V(v+a) e_{\mu}(z,v+a) = \sum_{\rho \in \calU_Q} \hat{V}_{\rho, \mu}(v) e_\rho(z,v+a) =  \sum_{\rho \in U_Q} \hat{V}_{\rho, \mu}(v) (\rho z)^a e_\rho(z,v), 
\]
where $\hat{V}_{\rho, \mu}(v)$ is the Fourier coefficient of $V$ on the orbit of $v\in W$ (see also \cite[Equation 4.5]{fillman2023algebraic}). 

We wish to solve for the Fourier coefficients $\hat{V}_{\rho, \mu}(v)$. Notice that 
\[
V(v+a) (\mu z)^a e_{\mu}(z,v) = \sum_{\rho \in U_Q} \hat{V}_{\rho, \mu}(v) (\rho z)^a e_\rho(z,v).
\]
Canceling out the $z^a$ on both sides, multiplying both sides by $\kappa^{-a}$ for some $\kappa \in \calU_Q$, and summing both sides over $a \in \ZZ^d$ such that $v+a \in W_Q$,  we obtain

\[\sum_{v + a \in W_Q} V(v+a) e_{\mu}(z,v) (\mu\kappa^{-1})^a = \sum_{v + a \in W_Q} \sum_{\rho \in U_Q} \hat{V}_{\rho, \mu}(v) (\rho \kappa^{-1})^a e_\rho(z,v). \]

When $\rho \neq \kappa$, the sum $\sum_{v+a \in W_Q}  \hat{V}_{\rho, \mu}(v) (\rho \kappa^{-1})^a e_\rho(z,v)$ is zero.  Thus the right-hand side simplifies to 
\[ 
\sum_{v+a \in W_Q} \sum_{\rho \in U_Q} \hat{V}_{\rho, \mu}(v) (\rho \kappa^{-1})^a e_\rho(z,v) =  \sum_{v+a \in W_Q} \hat{V}_{\kappa, \mu}(v) e_\kappa(z,v) = \lvert Q \rvert \hat{V}_{\kappa, \mu}(v) e_\kappa(z,v). 
\]
Replacing $\kappa$ with $\rho$, we see that
\[
\hat{V}_{\rho,\mu}(v) e_\rho(z,v) =   \frac{e_\mu(z,v) }{\lvert Q \rvert}\sum_{v+a \in W_Q} V(v+a) (\mu \rho^{-1})^{a}.
\]
Let \demph{$\hat{V}$} be the matrix representation of $V$ in the basis given by the functions of ~\eqref{eq:basis}; that is, $\hat{V}$ acts on the basis function $e_\rho(z,v)$ by
\[
\hat{V} e_\rho(z,v) = \sum_{\mu \in \calU_Q} \hat{V}_{\rho,\mu}(v) e_\rho(z,v) =  \sum_{\mu \in \calU_Q} \frac{e_\mu(z,v) }{\lvert Q \rvert}\sum_{v+a \in W_Q} V(v+a) (\mu \rho^{-1})^{a} \ \ \ \  \big( = Ve_{\rho}(z,v)\big).
\] 
This is a $Q \times Q$ block matrix with $m \times m$ entries, indexed the same as $\hat{\Delta}_E(z)$. Each $\hat{V}_{\rho,\mu}$ is an $m \times m$ diagonal matrix such that $(\hat{V}_{\rho,\mu})_{u,u} =  \hat{V}_{\rho,\mu}(u)$.
\medskip
\begin{Remark}~\label{rm:periodic}
If $V$ is also $\ZZ^d$-periodic, then
\[
\begin{split}
\hat{V}_{\rho,\mu}(v) e_{\rho}(z,v) & = \frac{1}{\lvert Q \rvert}\sum_{v + a \in W_Q} V(v+a)e_{\mu}(z,v) (\mu \rho^{-1})^{a} \\
& = \frac{1}{\lvert Q \rvert}\sum_{v+a \in W_Q} V(v)e_{\mu}(z,v) (\mu \rho^{-1})^{a}\\
&= \frac{V(v) e_{\mu}(z,v)}{\lvert Q \rvert} \sum_{v+a \in W_Q} (\mu \rho^{-1})^a .
\end{split}
\]
Thus, $\hat{V}_{\rho,\mu}(v) = V(v)$ when $\rho = \mu$ and is $0$ otherwise. That is, $\hat{V}$ is a diagonal matrix. \hfill$\diamond$\medskip
\end{Remark}

It follows that the $m \lvert Q \rvert \times m \lvert Q \rvert$ matrix $L_Q(z^Q)$ with respect to the basis $\{e_\mu(z,v) \mid \mu \in \calU_{Q} , v \in W\}$, is
\[
\textcolor{blue}{\hat{L}_Q(z)} =  \hat{V} + \hat{\Delta}_E(z).
\] 
Let $D_Q(z,\lambda) = \det(L_Q(z) - \lambda I)$ and $\hat{D}_Q(z,\lambda) = \det(\hat{L}_Q(z) - \lambda I)$. As $\hat{L}_{Q}(z,\lambda)$ is $L_Q(z^Q)$ after a change of basis, we see that
\[
D_Q(z^Q,\lambda) = \det(L_Q(z^Q) - \lambda I) = \det(\hat{L}_Q(z) - \lambda I) = \hat{D}_Q(z,\lambda).
\]
\begin{Example}\label{ex:expand_graphene} 
Let us continue Example~\ref{ex:floquet_graphene}. When we view the hexagonal lattice as $\ZZ^2$-periodic, as in the case of Figure~\ref{fig:periodic-graphs}, with a $\ZZ^2$-periodic potential $V$; we get the Floquet matrix
\[
L(x,y) = \begin{pmatrix}V(u){+}\alpha{+}\beta{+}\gamma & -\alpha{-}\beta x^{-1}{-}\gamma y^{-1}\\ -\alpha{-}\beta x{-}\gamma y & V(v){+}\alpha{+}\beta{+}\gamma\end{pmatrix}.
\]
Let $Q = (2,1)$ and let $V_Q$ be a $Q\ZZ$-periodic potential, then $L_Q(x,y)$ is given by the matrix
\[
\begin{pmatrix}V_Q(u){+}\alpha{+}\beta{+}\gamma & -\alpha{-}\gamma y^{-1} & 0 & {-}\beta x^{-1}\\ -\alpha{-}\gamma y & V_Q(v){+}\alpha{+}\beta{+}\gamma & {-}\beta & 0\\ 0 & {-}\beta & V_Q((1,0)+u){+}\alpha{+}\beta{+}\gamma &  {-}\alpha{-}\gamma y^{-1} \\{-}\beta x& 0 & {-}\alpha{-}\gamma y & V_Q((1,0)+v){+}\alpha{+}\beta{+}\gamma
\end{pmatrix}.
\]
If $V$ satisfies $V(u)=$ $\frac{V_{Q}(u) + V_{Q}((1,0)+u)}{2}$ and $V(v)= $$\frac{V_Q(v) + V_Q((1,0)+v)}{2}$, then $\hat{L}_Q(x,y)$ is a $2 \times 2$ block matrix with each entry a $2 \times 2$ matrix. Explicitly, 
\[
\hat{L}_Q(x,y) = \begin{pmatrix} L(x,y) & (\hat{V}_Q)_{1,-1}\\ (\hat{V}_Q)_{-1,1} & L(-x,y)\end{pmatrix}, \text{ where }
\]
\[
(\hat{V}_Q)_{1,-1} =  (\hat{V}_Q)_{-1,1} = \begin{pmatrix} \frac{V_Q(u) - V_Q((1,0)+u)}{2} &0\\0 & \frac{V_Q(v) - V_Q((1,0)+v)}{2}\end{pmatrix}.
\]

\vspace*{-0.9cm} \hfill$\diamond$\medskip
\end{Example}
\vspace*{0.4cm}

\subsection{Algebra and Discrete Geometry}\label{subsec:algdiscrete}
Let us recall some standard notions from algebra and discrete geometry. For more, see~\cite{Ewald}. A \demph{Laurent polynomial} \demph{$f$} $\in \CC[z^{\pm}]$ is a linear combination of finitely many monomials,
\[
f = \sum_{a \in \ZZ^d}c_{a}z^{a} \quad c_{a} \in \CC.
\] 
The \demph{support} of $f$, denoted by \demph{$\mathcal{A}(f)$}, is the set of indices $a \in \ZZ^{d}$ such that $c_{a}\neq0$, and the convex hull of $\mathcal{A}(f)$ is a polytope in $\RR^{d}$, called the \demph{Newton polytope} of $f$, denoted by \demph{$\newt{f}$}. The \demph{dimension} of a polytope is the dimension of the smallest affine subspace that contains it.

Let $P \subset \RR^{d}$ be a polytope and $w\in \ZZ^{d}$. Since $P$ is closed and bounded, the induced dot product $c\mapsto w\cdot c$ defines a linear function on $\RR^{d}$. Set \demph{$h_{P}(w)$} $\vcentcolon = \min\{w\cdot c\mid c\in P\}$. The \demph{face} $P_w$ \demph{exposed} by $w$ is a polytope and is given by $P_w = \{ p \in P \mid w \cdot p = h_{P}(w)\}$. For a polynomial $f = \sum_{a \in \calA(f)} c_a z^a$ and a face $F=\newt{f}$, the \demph{facial polynomial} of $F$ is \demph{$f|_F$} $\vcentcolon= \sum_{a \in F \cap \calA(f)} c_a z^a = f|_w$.

For a $d$-dimensional polytope $P$, the $(d{-}1)$-dimensional faces of $P$ are its \demph{facets} and the $0$-dimensional faces of $P$ are its \demph{vertices}. Since the vertices of $\newt{f}$ lie in $\mathcal{A}(f)\subseteq \ZZ^{d}$, this polytope is \demph{integral}. A monomial $z^a$ such that $\newt{z^a}$ is a vertex of $\newt{f}$ is an \demph{extreme monomial} of $f$.

The \demph{Minkowski sum} of polytopes $P$ and $K$ is a polytope, and is given by
\[
P+K \vcentcolon = \{p+k \ \ | \ \ p\in P, k\in K\}.
\]
 An integral polytope $P$ is \demph{indecomposable} if whenever $P$ is a Minkowski sum $P=K_{1}+K_{2}$, then one of $K_1$ or $K_2$ is a point. Otherwise, $P$ is \demph{decomposable}. An integral polytope $K$ is \demph{homothetic} to $P$ if there exists a rational number $r \geq 0$ and a point $a\in \ZZ^{d}$ such that $K = rP + a$. The integral polytope $P$ is \demph{only homothetically decomposable} if whenever $P = K_{1} + K_{2}$, one of the summands is homothetic to $P$. In this case, it follows that both summands are homothetic to $P$.

\begin{Example}
Let $H$ be a hyperplane of $\RR^d$, let $Y \subset H$ be a finite set,  and let $v$ be a point in $\RR^d \smallsetminus H$. The convex hull of $Y \cup \{v\}$ is a \demph{pyramid} (see Figures~\ref{fig:gpolytope} and~\ref{fig:densegraphpolytope} for examples).  Pyramids are only homothetically decomposable~\cite{Gao}. Thus given a polynomial of the form $g(x,y,z) = z^a + f(x,y)$, $\newt{g}$ is only homothetically decomposable. If $a=1$, then $\newt{g}$ is indecomposable.
\hfill$\diamond$\medskip
\end{Example}
\begin{figure}[ht]
    \centering
    \includegraphics{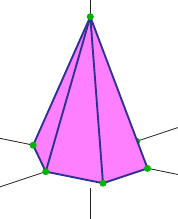}
    \caption{The Newton polytope of $D(x,y,\lambda)$ from Example~\ref{ex:floquet_graphene} is only homothetically decomposable.}
    \label{fig:gpolytope}
\end{figure}

Monomials with nonzero coefficients are the invertible elements in $\CC[z^{\pm}]$, forming the group of units $(\CC[z^{\pm}])^{\times}$. An element of a ring is irreducible if it is a nonzero nonunit that cannot be expressed as the product of two nonunits. Therefore, a Laurent polynomial $f$ is \demph{irreducible} if it is not a monomial, and whenever there exist Laurent polynomials $g$ and $h$ such that $f = gh$, then either $h$ or $g$ is a monomial. Given polynomials $f$ and $g$, the Newton polytope of their product $fg$ is given by $\newt{fg} = \newt{f} + \newt{g}$.
 A Laurent polynomial $f$ is \demph{only homothetically reducible} if it is not a monomial and if $f = gh$ implies that either $\newt{g}$ or $\newt{h}$ is homothetic to $\newt{f}$. Any irreducible Laurent polynomial is only homothetically reducible.

 \begin{Example}
 A polynomial with an only homothetically decomposable Newton polytope is only homothetically reducible. The converse is false. Consider the reducible polynomial $f(x,y) = (xy + x + y + 2)^2$. Here $\newt{f}$ is a square and thus can be decomposed into the two segments $\newt{1+y+y^2}$ and $\newt{1+x+x^2}$, neither of which is homothetic to $\newt{f}$. However, the polynomial is only homothetically reducible as each factor $xy+x+y+2$ is irreducible and $\newt{f} = 2\newt{xy+x+y+2}$. \hfill$\diamond$\medskip
\end{Example}

\section{From Only Homothetic Decomposability to Only Homothetic Reducibility}~\label{Sec:OHD2OHR}\label{sec:3-Bragg}

Only homothetic decomposability was considered in ~\cite{Shephard, McMullen}, where it is shown that if enough faces of a polytope are only homothetically decomposable, then the polytope itself must be only homothetically decomposable. Here we aim to prove an analogous result for only homothetically reducible Laurent polynomials.

 A \demph{strong chain of faces} of a polytope $P$ is a sequence of faces $F_1, \dots, F_n$ of $P$ \demph{of length $n$} such that for each $i$, $\dim F_{i}\cap F_{i+1}\geq 1$. 

\begin{Example}
Any adjacent triangular facets of a $3$-dimensional pyramid (such as those of Figure~\ref{fig:gpolytope}) share an edge, and thus give a strong chain of faces of length $2$.\hfill$\diamond$\medskip
\end{Example}

 We will abuse notation slightly throughout the following proofs. In particular, if $f,g,$ and $h$ are Laurent polynomials such that $f=gh$ and $F$ is a face of $\newt{f}$, we will write $f|_F = g|_F h|_F$ as the factorization of $f|_F$. Recall that there exists an inner normal $w \in \RR^d$ that exposes $F$ (and is such that $f|_F = f|_w$), and so really we mean that $g|_F = g|_w$ and $h|_F = h|_w$. We also assume that for any polytope $P$, $0\cdot P = \{\bfz\}$, where $\bfz$ is the vector of all zeroes.
\medskip
\begin{Remark}~\label{Rm:HomSimp}
If $f$ is only homothetically reducible then there exists Laurent polynomials $g,h$, and $r,t \in \QQ$ such that $f = g h$, $r \newt{f} =\newt{g}$, and $t\newt{f} = \newt{h}$. 

Indeed, by the definition of only homothetic irreducibility there exists $a_g$ and $a_h$ in $\ZZ^d$ so that $r \newt{f} + a_g =\newt{g}$, $t\newt{f} + a_h = \newt{h}$, and so $r \newt{f} + a_g + t\newt{f} + a_h  = \newt{f}$. It follows that $a_g + a_h = \bfz$, and thus there exists $g' = z^{a_h} g$ and $h' = z^{a_g} h$ such that  $r \newt{f} =\newt{g'}$ and $t\newt{f} = \newt{h'}$. \hfill $\diamond$
\end{Remark}

\begin{Lemma}\label{lem:homothetic}
 Let $f,g,$ and $h$ be Laurent polynomials and suppose that $f = gh$. Let $F_1$ and $F_2$ be faces of $\newt{f}$ with $\dim F_{1}\cap F_{2}\geq 1$ whose corresponding facial polynomials, $f|_{F_1}$ and $f|_{F_2}$, are only homothetically reducible. If $\newt{g|_{F_1}} = r\newt{f|_{F_1}}$ and $\newt{h|_{F_1}} = t\newt{f|_{F_1}}$ for some pair $r,t\in\QQ$, then $\newt{g|_{F_2}} = r\newt{f|_{F_2}}$ and $\newt{h|_{F_2}} = t\newt{f|_{F_2}}$.
\end{Lemma}

\begin{proof}
As $f=gh$, we have that $f|_{F_1} = g|_{F_1} h|_{F_1}$. As $f|_{F_1}$ is only homothetic reducible we have that $r \newt{f|_{F_1}} =\newt{g|_{F_1}}$ and $t\newt{f|_{F_1}} = \newt{h|_{F_1}}$ for some $r,t \in \QQ$.  Let $F' = F_1 \cap F_2$. As $F' \subset F_1$, it follows that $\newt{g|_{F'}} = r\newt{f|_{F'}}$ and $\newt{h|_{F'}} = t\newt{f|_{F'}}$. The polynomial $f|_{F_{2}}$ is only homothetically reducible and must agree with its restriction to $F'$; it follows that $r \newt{f|_{F_2}} =\newt{g|_{F_2}}$ and $t\newt{f|_{F_2}} = \newt{h|_{F_2}}$. 
\end{proof} 

\begin{Theorem}\label{thm:homothetic}
  Let $f$ be a Laurent polynomial, and suppose $f = gh$. If for each pair $(a,b)$ of distinct vertices of $\newt{f}$ there is a strong chain of faces $F_1, \dots, F_n$ such that $a \in F_1$, $b \in F_n$, and for each $F_{i}$, the corresponding facial polynomial $f|_{F_{i}}$ is only homothetically reducible, then $f$ is only homothetically reducible.
\end{Theorem}

\begin{proof}
    By Lemma~\ref{lem:homothetic}, there exist a pair of rational numbers $r,t \in \QQ$ such that $r \newt{f|_{F_i}} =\newt{g|_{F_i}}$ and $t\newt{f|_{F_i}} = \newt{h|_{F_i}}$ for all $i=1,\dots,n$. As $a \in F_1$ and $b \in F_n$, we have that $r \newt{f|_{a}} =\newt{g|_{a}}$, $t\newt{f|_{a}} = \newt{h|_{a}}$, $r \newt{f|_{b}} =\newt{g|_{b}}$, and $t\newt{f|_{b}} = \newt{h|_{b}}$. This is the case for all vertex pairs $(a,b)$ of $\newt{f}$. In particular, we may fix $a$ and let $b$ vary over the other vertices. As any vertex of $\newt{f}$ must come from the Minkowski sum of a pair of vertices $u,v$ where $u \in \newt{g}$ and $v \in \newt{h}$, and any vertex $u$ of $\newt{g}$ or $v$ of $\newt{h}$ must be a Minkowski summand for some vertex of $\newt{f}$; it follows that $r \newt{f} =\newt{g}$ and $t\newt{f} = \newt{h}$. 
\end{proof}

\begin{Corollary}\label{cor:irred}
Suppose that $f$ is only homothetically reducible. If there is a face $F$ of $\newt{f}$ such that $f|_F$ is irreducible, then $f$ is irreducible. 
\end{Corollary}
\begin{proof}
 Suppose that $f$ is only homothetically reducible and let $F$ be a face of $\newt{f}$ such that $f|_F$ is irreducible.  Suppose that $g,h$ are Laurent polynomials such that $f = gh$. As $f$ is only homothetically reducible, there exists $r,s \in \QQ$ such that  $r \newt{f} =\newt{g}$ and $t\newt{f} = \newt{h}$, and so for any face $F'$ of $\newt{f}$ we have that $r \newt{f|_{F'}} =\newt{g|_{F'}}$ and $t\newt{f|_{F'}} = \newt{h|_{F'}}$. Notice that $f|_F$ is irreducible and therefore, one of $g|_F$ or $h|_F$ is a monomial. As one of $h|_F$ or $g|_F$ must be a monomial (which by Remark~\ref{Rm:HomSimp} we can assume to be the constant monomial), either $t$ or $r$ is zero.
\end{proof}

\section{Expanded dispersion polynomials}\label{sec:2-structure}
Let $\Gamma$ be a $\ZZ^d$-periodic graph with fundamental domain $W$, let $L(z)$ be the Floquet matrix of a discrete periodic operator $L$ with a $\ZZ^d$-periodic labeling $(V,E)$ (that is, both $E$ and $V$ are $\ZZ^d$-periodic), and let $D(z,\lambda) \vcentcolon = \det(L(z)-\lambda I)$ be its dispersion polynomial. 

Fix $Q = (q_1,\dots, q_d) \in \NN^d$, and consider $\Gamma$ as a $Q\ZZ$-periodic graph with fundamental domain $W_Q$. For a $Q\ZZ$-periodic potential $V_{Q}$, let $L_Q(z)$ be the Floquet matrix of $L_{Q}$ acting on the $Q\ZZ$-periodic graph $\Gamma$ with fundamental domain $W_Q$ and with the labeling $(V_Q,E)$. Let $\hat{L}_Q(z)$ be the matrix obtained from the Floquet matrix $L_Q(z^{Q})$ after the change of basis~\eqref{eq:basis}, and let $\hat{V}$ be the matrix representing $V_Q$ after the change of basis in Section~\ref{coveringsubsec}. 

Recall that $D_Q(z,\lambda) = \det(L_Q(z) - \lambda I)$, $\hat{D}_Q(z,\lambda)= \det(\hat{L}_Q(z) - \lambda I)$, $D_Q(z^Q,\lambda) = \hat{D}_Q(z,\lambda)$, $\vert Q \vert := \prod_{i = 1}^d q_i$, and $\calU_Q \vcentcolon=$ $\prod_{i=1}^{d}\calU_{q_{i}}$, where $\calU_{q_{i}}$ is the multiplicative group of $q_{i}$th roots of unity. We will often write $D, D_Q,$ and $\hat{D}_Q$ in place of $D(z,\lambda), D_Q(z,\lambda)$ , and $\hat{D}_Q(z,\lambda)$ respectively.

We seek conditions on $D$ and $Q$ which imply that if $V_Q = V$ then $D_Q$ is irreducible. In this case, the $Q\ZZ$-periodic potential $V_Q$ is also $\ZZ^d$-periodic. 

 Suppose that $V_Q = V$. By Remark~\ref{rm:periodic}, $\hat{V}$ is given by a diagonal matrix when the potential $V_Q$ is $\ZZ^{d}$-periodic.  It follows that $\hat{D}_Q(z,\lambda)$ may be expressed in terms of $D(z,\lambda)$ as
\begin{equation}~\label{eq:expansion}
    \hat{D}_Q(z,\lambda) \vcentcolon = \det(\hat{L}_Q(z)- \lambda I) = \prod_{\mu \in \calU_Q} \det(L(\mu z) - \lambda I) = \prod_{\mu \in \calU_Q} D(\mu z, \lambda).
\end{equation}

Due to this expression, we have that $\newt{\hat{D}_Q} = {\lvert Q \rvert} \newt{D}$.
As $D_Q(z^Q,\lambda) = \hat{D}_Q(z,\lambda)$, $\newt{D_Q}$ is the polytope obtained after multiplying the $i$th coordinate of each point of $\newt{D}$ by $\frac{\lvert Q \rvert}{q_i}$. That is, $(a_1,\dots,a_{d},a_{d+1})$ is a vertex of $\newt{D}$ if and only if $(\frac{\lvert Q \rvert a_1}{q_1}, \dots,\frac{\lvert Q \rvert a_d}{q_d},$ $\lvert Q \rvert a_{d+1})$ is a vertex of $\newt{D_Q}$. Therefore, $w = (w_1, \dots, w_{d+1}) \in \ZZ^d$ exposes a face of $\newt{D}$ if and only if $w' = (q_1 w_1, \dots, q_d w_d,w_{d+1})$ exposes a face of $\newt{D_Q}$. We call $\newt{D_Q}$ a \demph{contracted $Q$-dilation} of $\newt{D}$ (a contracted $Q$-dilation is a $(\frac{|Q|}{q_1}, \dots, \frac{|Q|}{q_d}, |Q|)$-dilation in the sense of Section~\ref{SubSec:AEA}). We will often write $(D_Q)|_w$ for $(D_Q)|_{w'}$; similarly, if $F$ is the face of $\newt{D_Q}$ exposed by $w'$, then we will write $D|_F$ for the corresponding facial polynomial of $D$ and vice versa. 

The following lemma is an immediate consequence of ~\eqref{eq:expansion}. We include a proof for the reader's convenience. 

\begin{Lemma}\label{Lem:HomirRed}
Let $V_Q$ be the $\ZZ^d$-periodic potential $V$. Suppose that $D$ is only homothetically reducible, then $D_Q$ is only homothetically reducible.
\end{Lemma}
\begin{proof}
Suppose $D$ is only homothetically reducible and $D_Q = g(z,\lambda)h(z,\lambda)$, where $g(z,\lambda)$ is not a monomial.
As $D_Q(z^Q,\lambda) =\hat{D}_Q(z,\lambda)$, it suffices to show that $\newt{g(z^Q ,\lambda)}$ is homothetic to $\newt{\hat{D}_Q}$. 

By Remark~\ref{Rm:HomSimp}, as $D$ is only homothetically reducible, if $f_1,$ $\dots,f_l$ are its irreducible factors, then there exist $r_1,\dots,r_l \in \QQ$ such that $\newt{f_i}= r_i \newt{D}$. As $V_Q$ is $\ZZ^d$-periodic, it follows that $\newt{f_i}= \frac{r_i}{\lvert Q \rvert}  \newt{\hat{D}_Q}$.

By~\eqref{eq:expansion}, $g(z^Q,\lambda)h(z^Q,\lambda)$ $ = \hat{D}_Q =\prod_{\mu \in \calU_Q} D(\mu z, \lambda)$. 
Thus for some integer $s$, such that $0 < s \leq l \lvert Q \rvert$, $g(z^Q,\lambda) = \prod_{i=1}^s \kappa_i(z,\lambda)$, where each $\kappa_i(z,\lambda) = f_j(\mu z,\lambda)$ for some $j \in [l]$ and $\mu \in \calU_Q$ (noting that each $f_j(\mu z, \lambda)$ is an irreducible factor of $D(z^Q,\lambda)$). 
If $\kappa_i(z,\lambda) = f_j(\mu z,\lambda)$ then let $\chi_i = r_j$. 
We conclude that $\newt{g(z^Q,\lambda)} = \frac{\sum_{i = 1}^s \chi_i}{\lvert Q \rvert}  \newt{\hat{D}_Q}$, and thus we have $\newt{g(z,\lambda)} = \frac{\sum_{i = 1}^s \chi_i}{\lvert Q \rvert}  \newt{D_Q}$.
\end{proof}

\begin{Remark}\label{RM:HOMIRED}
Lemma~\ref{Lem:HomirRed} extends to facial polynomials and specializations (such as the specializations that define Fermi varieties).
That is, for a $\ZZ^d$-periodic potential, if $D|_w$ is only homothetically reducible, then so is $(D_Q)|_w$. Let $\lambda_0 \in \CC$. If $D(z,\lambda_0)$ is only homothetically reducible then so is $D_Q(z,\lambda_0)$. Finally, if $D|_w(z,\lambda_0)$ is only homothetically reducible, then so is $(D_Q)|_w(z,\lambda_0)$. Thus the results of this section extend to $D_Q(z,\lambda_0)$, as well as to $(D_Q)|_w(z,\lambda_0)$ and $(D_Q)|_w$. \hfill$\diamond$\medskip
\end{Remark}
The following lemma is considered folklore, and will provide us motivation. We include a proof for the reader's convenience. For $A \in \NN^d$, let $\demph{Q/A} := ( \frac{q_1}{a_1} , \dots, \frac{q_d}{a_d})$, and write $\demph{A\mid Q}$ if $a_i \mid q_i$ for all $i=1,\dots,d$.
\begin{Lemma}\label{Lem:Red}
 Suppose $A = (a_1,\dots,a_d) \in \NN^d$ such that  $A\mid Q$ and let $V_Q$ be an $A\ZZ$-periodic potential. If $D_Q$ is irreducible, then $D_A$ is irreducible.
\end{Lemma}
\begin{proof}
By way of contradiction, suppose that $D_A$ is reducible, that is, $D_A=f(z,\lambda) g(z,\lambda)$. The fundamental domain $W_Q$ is a $Q/A$-expansion of $W_A$, hence
\[
D_Q(z_1^{\frac{q_1}{a_1}}, \dots, z_d^{\frac{q_d}{a_d}}, \lambda) = \prod_{\mu \in \calU_{Q/A}} D_A (\mu  z, \lambda) = \prod_{\mu \in \calU_{Q/A}} f( \mu z, \lambda) g(\mu z, \lambda).
\]
By Lemma 3.1 of~\cite{FLM}, there exist $f'$ and $g'$ such that  \[f'(z_1^{\frac{q_1}{a_1}}, \dots, z_d^{\frac{q_d}{a_d}}, \lambda) = \prod_{\mu \in \calU_{Q/A}} f(\mu z , \lambda) \text{ \ \  and \ \ }g'(z_1^{\frac{q_1}{a_1}}, \dots, z_d^{\frac{q_d}{a_d}}, \lambda) = \prod_{\mu \in \calU_{Q/A}} g(\mu z , \lambda).\] Therefore $D_Q(z_1, \dots, z_d, \lambda) = f'(z_1, \dots, z_d, \lambda) g'(z_1, \dots, z_d, \lambda).$\end{proof}

By Lemma~\ref{Lem:Red}, if $D_Q$ is irreducible and $A\mid Q$, then $D_A$ is irreducible. For the remaining section, we assume $D$ is irreducible for the $\ZZ^d$-periodic potential $V$ and that $V_Q = V$. 
Let $\sigma  = \{ \sigma_1,\dots,\sigma_k \} \in \binom{[d]}{k}$ be a $k$-element subset of the set $[d] \vcentcolon = \{ 1, 2, \dots, d \}$. Let $\bar{\sigma} = [d] \smallsetminus \sigma$ be the complement of $\sigma$ in $[d]$. Define $\sigma \odot Q = (\sigma \odot q_1, \sigma \odot q_2, \dots, \sigma \odot q_d)$, where $\sigma \odot q_i = q_i$ if $i \in \sigma$, and $\sigma \odot q_j = 1$ if $j \not \in \sigma$. Let $D_{\sigma \odot Q}$ be the dispersion polynomial given by the discrete periodic operator $L$, with the $\ZZ^d$-periodic (and therefore $(\sigma \odot Q)\ZZ$-periodic) labeling $(V_Q,E)$ associated to the $(\sigma \odot Q)\ZZ$-periodic graph $\Gamma$ with fundamental domain given by the expansion $W_{\sigma \odot Q}$ of $W$.  This notation will allows us study irreducibility of the dispersion polynomial as we incrementally expand coordinate-wise from $D = D_{(1,\dots,1)}$ to $D_Q$. Lemma~\ref{Lem:Red} suggests this approach, as we know that if $D_Q$ is irreducible, then for any $k < d$, we must have that any $D_{\sigma  \odot Q}$ is irreducible for all $\sigma \in \binom{[d]}{k}$. Indeed, this leads us to the following theorem.
\medskip

\begin{Theorem}~\label{TH:1}
Fix a positive integer $k<d$ and suppose that $D_{\sigma \odot Q}$ is irreducible for all $\sigma \in \binom{[d]}{k}$. If no $k+1$ coordinates of $Q$ share a common factor, then $D_Q$ is irreducible.
\end{Theorem}
\begin{proof}
Assume that no $k+1$ coordinates of $Q$ share a common factor. Suppose there exist polynomials $g,h$, with $g$ not a monomial, such that
\[
D_Q = g(z,\lambda) h(z,\lambda).
\]
Reordering, if necessary, we may assume that $\sigma = [k]$. As $W_Q$ is an expansion of $W_{\sigma \odot Q}$,
\[
D_Q(z^{\bar{\sigma} \odot Q},\lambda)  = \prod_{\gamma \in \calU_{\bar{\sigma} \odot Q}} D_{\sigma \odot Q}(z_1,\dots, z_k, \gamma_1  z_{k+1}, \dots, \gamma_{d-k}  z_d, \lambda).
\]
As $D_{\sigma \odot Q}$ is irreducible, there exist $\gamma^1,\dots, \gamma^s \in  \calU_{\bar{\sigma} \odot Q}$ for some $s \geq 1$ such that
\[
g(z^{\bar{\sigma} \odot Q},\lambda) = \prod_{i = 1}^s D_{\sigma \odot Q}(z_1,\dots, z_k, \gamma^i_1 z_{k+1}, \dots, \gamma^i_{d-k} z_d, \lambda).
\]
Expand this so that 
\[
g(z^Q,\lambda) =\prod_{i = 1}^s \prod_{\mu \in \calU_Q} D(\mu_1 z_1,\dots,\mu_k z_k, \gamma^i_1  z_{k+1}, \dots, \gamma^i_{d-k} z_d, \lambda)
\]
can be written as a product of $S \vcentcolon = s \prod_{i =1}^k q_i$ irreducible polynomials. As $\sigma$ is arbitrary (that is, the same argument holds for any $\sigma \in \binom{[d]}{k}$ after reordering coordinates), the product $q_{\sigma_1} \cdots q_{\sigma_k}$ divides $S$ for all $\sigma \in \binom{[d]}{k}$.
By our assumption, no $k+1$ coordinates of $Q$ share a common factor. Therefore, if $p^a$ is a prime power that divides $|Q|$, there exists $\sigma \in \binom{[d]}{k}$ such that $p^a \mid q_{\sigma_1} \cdots q_{\sigma_k}$, and thus $p^a \mid S$. As $S$ is at most $|Q|$, it follows that $S = |Q|$, and so $h$ must be a monomial. 
\end{proof}

To apply Theorem~\ref{TH:1}, we need to find conditions that imply $D_{\sigma \odot Q}$ is irreducible for all $|\sigma|\geq 1$. Rather than depending strictly on $Q$, these conditions examine the reducibility of $D_Q$ in relation to the interplay between $Q$ and the support of $D$. We begin this discussion with a remark.
\medskip
\begin{Remark}~\label{rm:alpha} In what follows, we study how $D(z,\lambda)$ relates to $D(\mu z, \lambda)$ for $\mu \in \calU_Q$. In particular, we consider if there exists a $\mu \in \calU_Q$ such that $D(\mu z, \lambda)$ is given by $D(z,\lambda)$ up to multiplication by some constant. Notice that $D(z,\lambda)$ has a term that is constant as a polynomial in $z$, which is therefore invariant under the map $z \to \mu z$ (that is, $D(z,\lambda)$ necessarily has a term that is a constant or a power of $\lambda$). Therefore, we may always assume that if such a $\mu$ exists, then $D(\mu z, \lambda) = D(z,\lambda)$.
\hfill$\diamond$
\end{Remark}

Before stating these conditions in generality, we begin by building some intuition by studying the case when $d=1$. Suppose that $\sigma = \{ 1\}$, $z=z_1$, and that $q = q_1 > 1$. As $D(z,\lambda)$ is irreducible, we have that \[D
_{q}(z^{q}, \lambda) = \prod_{\mu \in \calU_{q}} D(\mu z, \lambda).
\]
If $D_{q}(z,\lambda) = g h$, where $g$ and $h$ are not monomials, then there exist $\mu_1, \dots, \mu_s \in \calU_{q}$, where $1 \leq s < q$, such that
\[
g(z^{q}, \lambda) = \prod_{i=1}^s D(\mu_i  z, \lambda).
\]
As $s < q$, there exists $\mu' \in \calU_{q}$ such that $\mu'\mu_1 \not \in \{ \mu_1, \dots, \mu_s \}$; indeed, such a $\mu'$ must exist otherwise $s = q$ and $D_{q}(z,\lambda)$ is irreducible as then $h$ must be a monomial. As multiplying $z$ by elements of $\calU_{q}$ does not change $g(z^{q},\lambda)$,  we have
\[
g(z^{q}, \lambda) = g((\mu' z)^{q}, \lambda) = \prod_{i=1}^s D(\mu' \mu_i  z, \lambda).
\]
As each $D(\mu z, \lambda)$ is irreducible, there is a $j \in \{1,\dots, s\}$ such that $D(\mu' \mu_1   z, \lambda) = D(\mu_j  z, \lambda)$ (see Remark \ref{rm:alpha}).

As $\mu'\mu_1 \neq \mu_j$, we have that $\demph{\hat{\mu}} = \mu' \mu_1 (\mu_j)^{-1}$ is not $1$, and thus we have a nontrivial element $\hat{\mu} \in \calU_{q}$ satisfying $D(\hat{\mu} z, \lambda) = D(z, \lambda)$. Since $D(\hat{\mu}  z, \lambda) = D(z, \lambda)$, if $\upsilon(z, \lambda)$ is a monomial term of $D(z,\lambda)$, then $\upsilon(\hat{\mu} z, \lambda) = \upsilon(z, \lambda)$. Thus if $D_{q}(z,\lambda)$ is reducible, then \demph{$\text{ord}(\hat{\mu})$}, the order of $\hat{\mu}$, must divide the exponent of $z$ in any term $\upsilon(z,\lambda)$ of $D(z,\lambda)$.

Let $b'$ be the greatest common divisor of the finite set of integers $\{r \mid (r,t) \in \calA(D(z,\lambda) \}$. As all terms of $D(z,\lambda)$ are invariant under the map $z \to \hat{\mu}z$, $\text{ord}(\hat{\mu})$ divides $b'$. As $\text{ord}(\hat{\mu})$ divides $q = |\calU_{q}|$, we see that $\gcd(q, b') \neq 1$. It follows that if $D(z,\lambda)$ is irreducible and $\gcd(q,b') =1$, then we have a contradiction and can conclude that $D_{q}(z,\lambda)$ is irreducible (as our assumption that $D_{q}(z,\lambda)$ is reducible implies that $\gcd(q,b') \neq 1$).

Indeed, if $\gcd(q, b') = 1$, then we have that $\gcd(\text{ord}(\hat{\mu}), b') =1$. By Euclid's algorithm and the definition of $b'$, there must exist $z^{r_1}\lambda^{t_1}, \dots, z^{r_l}\lambda^{t_l}$ as integer powers of monomials, or the reciprocals of monomials, that appear as a term with nonzero coefficient in $D(z,\lambda)$ with $\sum r_i = b'$. Therefore, as $\hat{\mu}^{b'} \neq 1$, we see that \begin{equation}~\label{eq:Divstuff}
    \hat{\mu}^{b'}  z^{r_1 + \dots + r_l} \lambda^{t_1 \dots t_l} \neq   z^{r_1 + \dots + r_l} \lambda^{t_1 \dots t_l},\end{equation} and so we cannot have $(\hat{\mu} z)^{r_i}\lambda^{t_i} = z^{r_i}\lambda^{t_i}$ for all $i \in [l]$. This contradicts the assumption that $D_{q}(z,\lambda)$ is reducible, as if it were then, as discussed, all terms of $D(z,\lambda)$ would be invariant under the map $z \to \hat{\mu} z$.

To state the more general case, we first need to introduce a definition that will allow us to identify the values $\text{ord}(\hat{\mu})$ can take for $D_{\sigma \odot Q}$ to be reducible.
\begin{Definition} Let $\sigma \in \binom{[n]}{k}$ for some $k \in [n]$ and let $j \in \sigma$. Let $B$ be the collection of $b$ such that there is a vector in the integer span of $\calA(D)$ that is $b$ in the $j$th coordinate and $0$ for every other coordinate $i \in \sigma$. $B$ forms an ideal of $\ZZ$ and is therefore principal. Define $\demph{\Span_{j, \sigma}(D)}$ to be a generator of $B$ (equivalently, the greatest common divisor of the elements in $B$).\hfill$\diamond$\medskip

\end{Definition}
 If $Q = q_1$, then $\Span_{1, \{1 \}}(D) = b'$ (where $b'$ is from the discussion of the one-dimensional case). In general, $D_{\sigma \odot Q}$ can factor only if $\text{ord}(\hat{\mu})$ divides $\Span_{1, \sigma}(D)$ (as otherwise the same situation as \eqref{eq:Divstuff} arises). We now present some examples.
 \medskip
\begin{Example}~\label{ex:27}
Consider the polynomial $f(z_1,z_2,\lambda) = z_1^2 z_2^2 + \lambda z_1^4 + \lambda^3$. Suppose there is a $\mu_1 \in \TT$ such that $f(\mu_1 z_1, z_2,\lambda) = c f(z_1, z_2,\lambda)$ for some $c \in \CC$. As every term must be invariant under $z_1 \to \mu_1 z_1$, $c = 1$ because $\lambda^3$ is invariant with respect to this change of variables. By definition, $\Span_{1,\{1\}}(f(z_1,z_2,\lambda))=2$. Thus $\mu_1^2 = 1$, that is $\mu_1 = \pm 1$. This agrees with the fact that $\mu_1^2z_1^2z_2^2 = z_1^2z_2^2$.

In this same case, $\Span_{1,\{1,2\}}(f(z_1,z_2,\lambda)) = 4$. Given $\mu_1$ and $\mu_2$ in $\TT$, where $f(\mu_1 z_1, \mu_2 z_2,\lambda)$ $=$ $  c f(z_1, z_2,\lambda)$, then $c = 1$. As $\lambda z_1^4$ is independent of $\mu_2$, the order of $\mu_1$ must divide $4$.

Finally consider $h(z_1,z_2,\lambda) = z_1^{-3} z_2^2 + z_1^2 z_2^{-1} + \lambda$. Assume $\mu_1$ and $\mu_2$ are in $\TT$ such that $h(\mu_1 z_1, \mu_2 z_2,\lambda) = c h(z_1, z_2,\lambda)$. Again, $c = 1$. We have $(z_1^{-3} z_2^2)(z_1^2z_2^{-1})^2 = z_1$. Therefore, we find that $\Span_{1,\{1,2\}}(h(z_1,z_2,\lambda)) = 1$. As $\mu_1^{-3} z_1^{-3} \mu_2^2z_2^2 =z_1^{-3} z_2^2$ and $\mu_1^2z_1^2\mu_2^{-1}z_2^{-1} = z_1^2z_2^{-1}$, it follows $z_1 = (\mu_1^{-3} z_1^{-3} \mu_2^2z_2^2)(\mu_1^2z_1^2\mu_2^{-1}z_2^{-1})^2 = \mu_1 z_1$; we conclude that $\mu_1 = 1$.
\hfill$\diamond$
\end{Example}

\begin{Remark}
If $\sigma' \subseteq \sigma$ then $\Span_{j,\sigma'}(D)$ divides $\Span_{j,\sigma}(D)$.
\hfill$\diamond$\medskip
\end{Remark}

We now state the general case. Recall Remark~\ref{rm:alpha}; that is, we assume that if there exists a $\mu \in \calU_Q$ such that $D(\mu z, \lambda) = c D(z,\lambda)$, then $c = 1$.

\begin{Lemma}\label{lem:coprime}
Let $V_Q$ be a $\ZZ^d$-periodic potential. Suppose that there exists $\sigma' \in \binom{[d]}{k-1}$, where $1 \leq k \leq d$, such that $D_{\sigma' \odot Q}$ is irreducible. Let $\sigma = i \cup \sigma'$ for some $i \not\in \sigma'$. If $q_i$ is coprime to $b = \Span_{i,\sigma}(D)$, then $D_{\sigma \odot Q}$ is irreducible. 

\end{Lemma}
\begin{proof}
After reordering we may assume that $i = 1$ and $\sigma = \{1,2, \dots, k\}$. 
By way of contradiction, suppose that $D_{\sigma \odot Q}$ is reducible with factor $g$ that is not a monomial, but $\gcd(q_1,b) = 1$. Let $\sigma' = \sigma\smallsetminus\{1\}$. Then
\[
D_{\sigma \odot Q}(z_{1}^{q_{1}},z_2,\dots, z_d,\lambda) = \prod_{\mu\in \calU_{q_{1}}} D_{\sigma' \odot Q}(\mu z_{1},z_2,\dots,z_d,\lambda).
\]
As each $D_{\sigma' \odot Q}(\mu z_{1},z_2,\dots,z_d,\lambda)$ is irreducible, $g$ must have the following factorization,
\[
g(z_{1}^{q_{1}},z_2,\dots,z_d, \lambda)=\prod_{i=1}^{s}D_{\sigma' \odot Q}(\mu_{i} z_{1},z_2,\dots,z_d,\lambda),
\]
where $\mu_i \in U_{q_1} $ and $1 \leq s < q_1$; that is, a nonempty proper subset of the irreducible factors of $D_{\sigma \odot Q}(z_{1}^{q_{1}},z_2,\dots, z_d,\lambda)$ must appear as the irreducible factors of $g(z_{1}^{q_{1}},z_2,\dots,z_d, \lambda)$. As $s < q_1$, there exists $\mu' \in \calU_{q_1}$ such that $\mu' \mu_1 = \hat{\mu} \not\in \{\mu_1, \mu_2, \dots , \mu_s\}$. Notice we have the following two factorizations, 
\begin{equation*}
\begin{split}
g(z_{1}^{q_{1}},\dots,z_k^{q_k},z_{k+1},\dots,z_d, \lambda) & = \prod_{i=1}^{s}\prod_{\gamma\in \calU_{(1, q_2,\dots,q_k)}}D(\mu_{i} z_{1},\gamma_{2} z_{2},\dots, \gamma_{k}z_{k},z_{k+1},\dots,z_d, \lambda),\\
g((\mu' z_{1})^{q_{1}},\dots,z_k^{q_k},z_{k+1},\dots,z_d, \lambda) & = \prod_{i=1}^{s}\prod_{\gamma \in \calU_{(1, q_2,\dots,q_k)}} D(\mu'\mu_{i} z_{1},\gamma_{2} z_{2},\dots, \gamma_{k} z_{k},z_{k+1},\dots,z_d, \lambda).
\end{split}
\end{equation*}

As each $D(\mu z,\lambda)$ is irreducible and \[g(z_{1}^{q_{1}},\dots,z_k^{q_k},z_{k+1},\dots,z_d, \lambda) = g((\mu' z_{1})^{q_{1}},\dots,z_k^{q_k},z_{k+1},\dots,z_d, \lambda),\] these two factorizations are the same. For $\hat{\mu}$ and a given $\gamma \in \calU_{(1, q_2,\dots,q_k)}$, there exists $\mu_l \in \{\mu_1, \mu_2, \dots \mu_s\}$ and $\gamma' \in \calU_{(1, q_2,\dots,q_k)}$ with
\[
D(\hat{\mu} z_{1},\gamma_{2} z_{2},\dots, \gamma_{k} z_{k},z_{k+1}, \dots, z_d, \lambda) =  D(\mu_l  z_{1},\gamma'_{2} z_{2},\dots, \gamma'_{k} z_{k},z_{k+1}, \dots, z_d, \lambda).
\]
Let $\tilde{\mu} = \mu_l (\hat{\mu})^{-1}$. Notice that, as $\hat{\mu} \not\in \{ \mu_1,\dots, \mu_s\}$, $\tilde{\mu} \neq 1$. In particular, $\text{ord}(\tilde{\mu})$ is an integer greater than $1$ that divides $q_1$. If $z_2, \dots, z_k$ do not appear in a monomial $\upsilon(z_1,z_{k+1},\dots z_d,\lambda)$ with support in the integral span of $\calA(D)$ then $\upsilon(\tilde{\mu} z_1,z_{k+1},\dots z_d,\lambda) = \upsilon(z_1,z_{k+1},\dots z_d,\lambda)$. Since $\gcd(q_1,b)=1$, by the definition of $\Span_{1,\{1,2,\dots, k\}}(D)$ ($= b$), there exists a term $\upsilon(z_1,z_{k+1},\dots z_d,\lambda)$ of $D$ that is not invariant under the map $z \to \tilde{\mu}z$, a contradiction.
\end{proof}
\medskip

\begin{Remark}~\label{rm:alpha2} 
Lemma~\ref{lem:coprime} can also apply to $D(z,\lambda_0)$, $D|_F$, and $D|_F(z,\lambda_0)$ (that is, the polynomials defining the Fermi varieties and facial polynomials); however, there may be $\mu \in \calU_Q$ such that $D(\mu z, \lambda_0) = c D(z,\lambda_0)$ for some $c \in \TT \smallsetminus \{ 1 \}$. In these cases, one can still obtain the assumption of Remark~\ref{rm:alpha} by only taking partial expansions of the fundamental domain. For example, given $\sigma \in \binom{[d]}{k}$, $D(z,\lambda)$ may have a term that is a constant as a polynomial in $\CC[z_{\sigma_1}, z_{\sigma_2},\dots, z_{\sigma_k}]$ (that is, a term in $\CC[z_{\bar{\sigma}_1}, z_{\bar{\sigma}_2},\dots, z_{\bar{\sigma}_{d-k}},\lambda]$). It follows that if $\mu \in \calU_{\sigma \odot Q}$ and $D(\mu z, \lambda_0) = c D(z,\lambda_0)$, then $c = 1$.
\hfill$\diamond$
\end{Remark}

\medskip
\begin{Remark}~\label{remark:past}
From the proof of Lemma~\ref{lem:coprime} we can recover a version of ~\cite[Lemma 3.4]{FLM}. In particular, suppose that for all $\mu \in \calU_Q$ one has $D(\mu z,\lambda) \neq D(z,\lambda)$ (using the assumption that $D$ has a term that is constant as a polynomial in $z$). This condition essentially encapsulates condition (A2) of ~\cite{FLM}, which is an assumption of ~\cite[Lemma 3.4]{FLM}. Under this condition, notice that if $g | D_Q$ and $g$ is not a monomial, then we must have that $D(\mu z, \lambda) | g(z^Q,\lambda)$ for every $\mu \in \calU_Q$; but then $g(z^Q,\lambda) = \prod_{\mu \in \calU_Q} D(\mu z, \lambda) = D_Q(z^Q,\lambda)$. Thus Lemma~\ref{lem:coprime} essentially gives us conditions for when (A2) holds when expanding the fundamental domain along a single coordinate axis (allowing us to apply the argument of ~\cite[Lemma 3.4]{FLM}). 

More generally, we have the following. Suppose that there exists $\sigma' \subsetneq \sigma$ such that $D_{\sigma' \odot Q}$ is irreducible and contains a constant term as a polynomial in $z$. Without loss of generality, let $\sigma = \{1,\dots, l\} \cup \sigma'$ where $\{1,\dots, l \} \subseteq \overline{\sigma'}$, and let $\calU_{Q'} = \calU_{\{1,\dots, l \} \odot Q}$. If \[D_{\sigma' \odot Q}(z_1,\dots, z_d, \lambda) \neq D_{\sigma' \odot Q}(\mu_1 z_1,\dots, \mu_l z_l, z_{l+1}, \dots, z_d, \lambda) \text{ for any }\mu \in \calU_{Q'},\] then $D_{\sigma \odot Q}$ is irreducible.

We avoid further discussions of this general criteria, as our goal is to present practically verifiable conditions on $D$ that enable us to conclude irreducibility for $D_{Q}$. \hfill $\diamond$
\end{Remark}

We will use the following corollary often in the examples of Section~\ref{SubSec:AEA}.

\begin{Corollary}\label{cor:coprime}
Let $V_Q$ be the $\ZZ^d$-periodic potential $V$, and suppose that $D$ is irreducible. If there exist terms $z_1^{a_1}, \dots, z_d^{a_d}$ with nonzero coefficients in $D$, then $D_{Q}$ is irreducible for all $Q$ such that $\gcd(q_i,a_i) = 1$ for all $i$. 
\end{Corollary}
\begin{proof}
    Notice that no matter our choice of $i$ and $\sigma \subseteq [d]$, we have that $\Span_{i,\sigma}(D) | a_i$ and thus $\gcd(q_i, \Span_{i,\sigma}(D)) = 1$. Starting with the fact that $D = D_{\{ \} \odot Q}$ is irreducible and then consecutively applying Lemma~\ref{lem:coprime}; we see that at each step, as $\gcd(q_i, \Span_{i,\sigma}(D)) = 1$, we have that $D_{\sigma \odot Q}$ is irreducible for each $\sigma \subseteq [d]$.
\end{proof}

We say that a facial polynomial $D_Q|_F$ is \demph{$\ZZ^d$-periodic} if there exists a $\ZZ^d$-periodic potential $V'$ corresponding to a dispersion polynomial $D_Q'$ such that $D_Q|_F = (D_Q')|_F$. 

Suppose that for a $Q\ZZ$-periodic potential $V_Q$ the facial polynomial $D_Q|_F$ is $\ZZ^d$-periodic due to the existence of a $\ZZ^d$-periodic $V'$. By Remark ~\ref{RM:HOMIRED}, if $D|_F$ is only homothetically reducible for $V'$, then $D_Q|_F$ is only homothetically reducible. By Theorem~\ref{thm:homothetic}, we obtain the following corollary.

\begin{Corollary}\label{cor:ZdPeriodic}
Suppose that for every facet $F$ of $\newt{D_Q}$, except possibly one, $D_Q|_F$ is $\ZZ^d$-periodic via the existence of a $\ZZ^d$-periodic potential $V_F$. If, for each $F$, $D|_F$ is only homothetically reducible for $V_F$, then $D_Q$ is only homothetically reducible.
\end{Corollary}

\section{Applications}\label{sec:4-applications} ~\label{SubSec:AEA}

We conclude with examples of discrete periodic operators associated to various families of periodic graphs which have irreducible Bloch varieties and irreducible Fermi varieties. We assume that all edge labels are nonzero. We begin by introducing useful definitions and notation to be used in the examples to come.

For a Laurent polynomial $f$, a facial polynomial $f|_F$ for a face $F$ of $\newt{f}$ is \demph{potential-independent} if the potential, treated as a finite vector of indeterminates, does not appear in the coefficients of $f|_F$. If a facial polynomial $D_Q|_F$ is potential-independent, then $D_Q|_F$ is $\ZZ^d$-periodic via the zero potential. A face $F$ of $\newt{D}$ (and its facial polynomial $D|_F$) is \demph{apical} if $F$ contains the apex of $\newt{D}$,  $(0,\dots, 0, m)$. 

Let $S_{m}$ be the symmetric group on $m$ elements, \demph{$L(z,\lambda)$} $\vcentcolon=$ $L(z) - \lambda I$, and let $L_{i,j}(z,\lambda)$ be the $(i,j)$ entry of $L(z,\lambda)$. 
For $w \in \ZZ^{d+1}$, we call $w \cdot (a,l)$ the \demph{weight} of the term $z^a \lambda^l$ with respect to $w$. 
The monomial terms in \demph{$\tau L(z,\lambda)$} $\vcentcolon = \prod_{i=1}^{m} L_{i,\tau(i)}(z,\lambda)$ are said to be terms \demph{produced} by $\tau$. Notice that in this way, $D(z,\lambda) =  \sum_{\tau \in S_m} \text{sgn}(\tau) \tau L(z,\lambda) $. We say a permutation $\tau \in S_m$ \demph{contributes} to terms of $D|_w$ if $\calA(\tau L(z,\lambda)) \cap \calA(D|_w)\neq \emptyset$.  We say $\tau$ is \demph{nonzero} if $\tau L(z,\lambda) \neq 0$. \medskip

\subsection{Bloch Varieties}~\label{subsec:BlochV}
 We say a $\ZZ^d$-periodic graph is a \demph{1-vertex graph} if it has a single vertex orbit with respect to its $\ZZ^d$-action. For $d\geq 1$, let $\Gamma$ be a $1$-vertex $\ZZ^d$-periodic graph.  Let $L$ be a discrete periodic operator associated to $\Gamma$. By \cite{Gao}, $D$ is irreducible as $\newt{D}$ is a pyramid of height $1$. Any apical facet of $\newt{D}$ is also a pyramid of height $1$ and thus has an irreducible facial polynomial. 
\begin{figure}[ht]
    \centering
       \begin{subfigure}{0.35\textwidth}
    \includegraphics[scale=1.4]{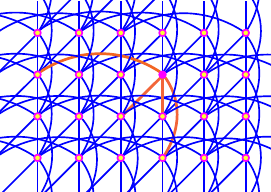}
     \end{subfigure}
     \hspace{1in}
      \begin{subfigure}{0.35\textwidth}
    \includegraphics[scale=1.4]{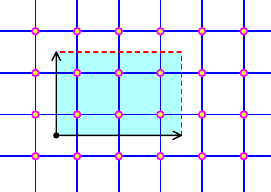}
     \end{subfigure}

    \caption{Two $2$-dimensional $1$-vertex graphs. On the left, the
     orange edges are representatives of the edge orbits. On the right, the square lattice with a highlighted $(3,2)$-expansion of the fundamental domain is depicted.}
       \label{fig:grid}
\end{figure}
Suppose $Q \in \NN^d$ and let $V$ be a $Q\ZZ$-periodic potential. For any nonbase face, as $\hat{L}_Q(z,\lambda)$ is only nonconstant along the main diagonal, it has only one contributing permutation (namely, the identity permutation) with each entry contributing terms of the same negative weight. 
We conclude that for every nonbase face $F$, the facial polynomial $\hat{D}_Q|_F$, and therefore  $D_Q|_F$, is potential-independent. As the faces are potential-independent (see the discussion before Lemma~\ref{Lem:HomirRed}), $\newt{D_Q}$ is a contracted $Q$-dilation of $\newt{D}$, and therefore a pyramid. Therefore, $D_Q$ is only homothetically reducible. To conclude that $D_Q$ is irreducible, we must show that one facial polynomial is irreducible.
\medskip
\begin{Example}~\label{1vBloch1}
Consider a $\ZZ^d$-periodic $1$-vertex graph $\Gamma$, where $d \geq 1$, such that $D$ has a facial polynomial $D|_F$ with the extreme monomials $z_1^{a_1}, z_2^{a_2}, \dots, z_d^{a_d}$. Two  examples of $1$-vertex graphs with this property are shown in Figure~\ref{fig:grid}.  If $q_i$ is coprime to $a_i$ for each $i$, then by Corollary~\ref{cor:coprime} $D_Q|_F$ is irreducible. Thus, by Corollary~\ref{cor:irred}, it follows that $D_Q(z,\lambda)$ is irreducible for all potentials. 

For example, consider the left-hand graph of Figure~\ref{fig:grid}. We have that $\newt{D}$ has a face with the extreme monomials $z_1^{3}$ and $z_2^2$; thus if $q_1$ is coprime to $3$ and $q_2$ is coprime to $2$ then $D_Q(z,\lambda)$ is irreducible.
\hfill$\diamond$\medskip
\end{Example}

\begin{Example}~\label{1vBloch2}
Let $d \geq 1$. Take any $1$-vertex $\ZZ^d$-periodic graph with at least one edge. Pick an apical facet $F$ of $\newt{D}$. Notice that there must be some monomial $z^a$ occurring as a term of $D|_F$ with a nonzero coefficient, for some $a~(\neq \bfz) \in \ZZ^d$.  Due to this, the collection of $Q = (q_1,\dots, q_d) \in \NN^d$ such that $D_{\{i\} \odot Q}$ is irreducible for all $i \in [d]$ is infinite. In particular, this set contains the $Q \in \NN^d$ such that $\gcd(a_i,q_i) = 1$; as $\Span_{i,\{i\}}(D|_F)$ must divide $a_i$, we have that $D_{\{i\} \odot Q}$ is irreducible by Lemma~\ref{lem:coprime}. Moreover, we consider the infinite set of $Q \in \NN^d$ such that $\gcd(a_i,q_i) = 1$ and the coordinates of $Q$ are pairwise coprime. Given a $Q$ in this infinite subset, we see that $D_Q|_F$ is irreducible by Theorem~\ref{TH:1}. Thus $D_Q$ is irreducible for all potentials.
\hfill$\diamond$
\end{Example}

Let us briefly compare our methods and results with those of ~\cite{FLM}.
\medskip
\begin{Remark}~\label{rm:compare}
The results of these last two examples overlap with the results and methods of ~\cite{FLM}.  In particular, if $F$ is a facet that is not the base, then $D_Q$ is irreducible if $D|_F(\mu z, \lambda) \neq D|_F(z,\lambda)$ for all $\mu \in \calU_Q$ (this is what they refer to as condition (A2), see Remark~\ref{remark:past}). In ~\cite{FLM} $1$-vertex graphs were considered, and thus checking whether (A2) is satisfied is sufficient to conclude irreducibility of $D_Q$; as this condition implies irreducibility of the facial polynomial $D_Q|_F$ and only homothetic reducibility immediately follows from the fact that the Newton polytope are pyramids (this is essentially \cite[Lemma 3.6]{FLM}). 

A distinct difference between these methods is that we only require that any facial polynomial be irreducible, whereas in \cite{FLM} they always fix the face given by $w = (1,\dots, 1, -1)$. For example, in \cite{FLM} it is concluded that the dispersion polynomial obtained from the Schr\"odinger operator associated to the Harper lattice is irreducible for all $(q_1,q_2)\ZZ$-periodic potentials when $q_1$ and $q_2$ are coprime, but choosing another facial polynomial (for example, corresponding to $w = (-1,0,-1)$) reveals that $q_1$ and $q_2$ do not need to be coprime.  \hfill $\diamond$
\end{Remark}
\medskip
\begin{Example}~\label{GrapheneBloch}
The $d$th member of the \demph{hexagonal-diamond family} is a $\ZZ^d$-periodic graph $\Gamma$ with fundamental domain $W$ given by two vertices $u$ and $v$ and edges $(u,v), (u,v-\varepsilon_i)$, and $(\varepsilon_i{+}u,v)$, where $\varepsilon_{i}$ is the $i$th standard basis vector of $\RR^{d}$. Due to periodicity, $E_{(u,v-\varepsilon_i)} = E_{(\varepsilon_i{+}u,v)}$. The $d=2$ and $d=3$ members are shown in Figure~\ref{fig:graphdiamond}.
Let $\gamma_i$ and $\alpha$ be $\ZZ^d$-periodic edge labels of $\Gamma$. For a $\ZZ^d$-periodic potential, a discrete periodic operator associated to $\Gamma$ has the Floquet matrix  
\[
L(z,\lambda) =\begin{pmatrix} \alpha + \sum_{i = 1}^d \gamma_i + V(u)-\lambda & -\alpha-\gamma_1 z_{1}^{-1}-\dots -\gamma_d z_d^{-1} \\ -\alpha-\gamma_1 z_{1}-\dots -\gamma_d z_d & \alpha + \sum_{i = 1}^d \gamma_i + V(v) - \lambda\end{pmatrix}. 
\]
\begin{figure}         \centering
       \begin{subfigure}{0.35\textwidth}
    \includegraphics[scale=0.32]{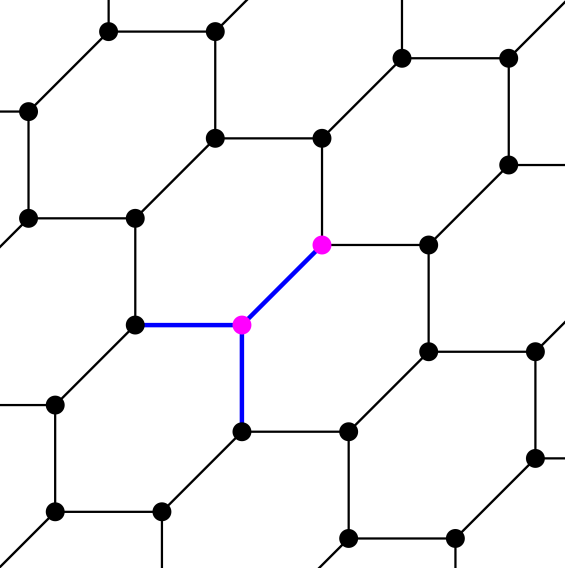}
     \end{subfigure}
     \hspace{1in}
      \begin{subfigure}{0.35\textwidth}
    \includegraphics[scale=0.2]{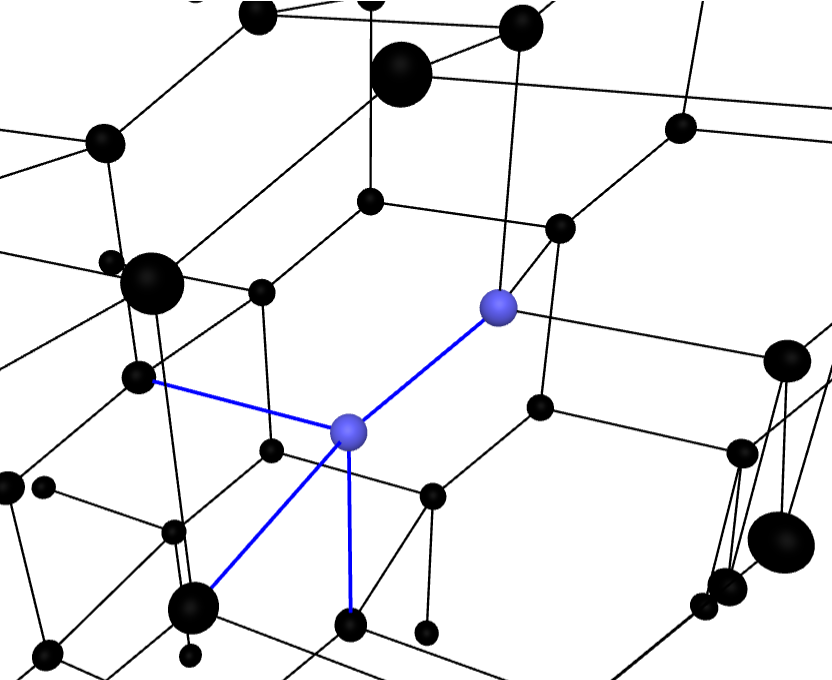}
     \end{subfigure}\caption{Local realizations of the hexagonal (left) and the diamond lattice (right). The purple vertices and blue edges are representatives of vertex and edge orbits of the graphs, respectively.}
    \label{fig:graphdiamond}
\end{figure}

The apical facial polynomials of $D$ are potential-independent and irreducible. Let $Q\in \NN^{d}$. We show that the apical facial polynomials of $\hat{D}_Q$, and thus $D_Q$, are potential-independent. Consider $\hat{L}_Q(z,\lambda)$ for a $Q\ZZ$-periodic potential $V$ $\vcentcolon$
\[
\hat{L}_Q(z,\lambda) = \begin{pmatrix}
    L(\mu_{1} z,\lambda) & \begin{matrix} \hat{V}_{12,{1}} & 0\\ 0 & \hat{V}_{12,2} \end{matrix} & \cdots & \begin{matrix} \hat{V}_{1\lvert Q\rvert,{1}} & 0\\ 0 & \hat{V}_{1\lvert Q\rvert,{2}} \end{matrix} \\ 
    \begin{matrix} \hat{V}_{21,1} & 0\\ 0 & \hat{V}_{21,2} \end{matrix} & L(\mu_{2} z,\lambda) & \cdots & \vdots \\
    \vdots &  & \ddots & \\
    \begin{matrix} \hat{V}_{\lvert Q\rvert1,{1}} & 0\\ 0 & \hat{V}_{\lvert Q\rvert1,{2}} \end{matrix} & \cdots & & L(\mu_{\lvert Q\rvert} z,\lambda)\\
\end{pmatrix},
\]
where $\hat{V}_{ij,k} = (\hat{V}_{\mu_i,\mu_j})_{k,k}$ and $L(z,\lambda)$ has $\tilde{V}(u) = \frac{1}{\lvert Q \rvert} \sum_{\omega | \omega + u \in W_Q} V(\omega + u)$ as its $\ZZ^d$-periodic potential. Viewing a nonzero permutation $\tau$ as a collection of paths (or directed cycles) through the matrix, we say that $\tau$ \demph{leaves} the main block-diagonal if there exists an $i \in [2Q]$ such that $(i,\tau(i))$ is an entry not belonging to the block-diagonal, i.e. for some $i$ we have that 
\medskip

\centerline{$(i,\tau(i)) \not\in \{(2n-1,2n-1),(2n-1,2n),(2n,2n-1),(2n,2n) \mid 1 \leq n \leq \lvert Q \rvert \}.
$}
\medskip
Without loss of generality, let $F$ be the apical facet exposed by $(-2,\dots,-2,-1)$. Every diagonal $L(\mu_k z, \lambda)$ can contribute at most $\lambda^2$ or $z_i$ to a term of $\tau\hat{L}_Q(z,\lambda)$, and so for $\tau$ to contribute to $\hat{D}_{Q}|_F$, every $L(\mu_k z, \lambda)$ must contribute either $\lambda^2$ or $z_i$, but not both. If $\tau$ leaves the main block-diagonal, then for some $k \in [|Q|]$ the permutation of $\tau$ cannot involve entries in the bottom row or left-hand column of $L(\mu_{k} z, \lambda)$; but then $L(\mu_{k} z, \lambda)$ cannot contribute either $\lambda^2$ or $z_i$ to $\tau\hat{L}_Q(z,\lambda)$, and so we must have that $\tau$ does not contribute to $\hat{D}_{Q}|_F$. It follows that $\hat{D}_{Q}|_F = \prod_{i=1}^{\lvert Q \rvert} \det(L(\mu_i z,\lambda))|_F$.

Thus $D_Q$ has only potential-independent, and thus only homothetically reducible, apical facial polynomials. As $D|_F$ is irreducible and has extreme monomials $z_1, z_2, \dots , z_d ,$ and $\lambda$; it follows from Corollary~\ref{cor:coprime} that $D_Q|_F$ is irreducible. Thus, by Corollary~\ref{cor:irred}, $D_Q$ is irreducible.
\hfill$\diamond$
\end{Example}

A $\ZZ^d$-periodic graph $\Gamma$ is \demph{dense} if there is a fundamental domain $W$ of $\Gamma$ such that whenever $a \in \calA(W) \neq \emptyset$, the union of $W$ and $W+a$ induces a complete graph. 

\medskip

\begin{Example}~\label{ex:DenseBloch} Consider a $\ZZ^d$-periodic dense graph $\Gamma$. As $\Gamma$ is dense, by \cite[Lemma 4.3]{FS}, $\newt{D}$ and its apical facets are pyramids for a generic labeling. For $Q \in \NN^d$, it is straightforward to deduce that any nonbase facial polynomial $D_Q$ is potential-independent and thus only homothetically reducible (such as in Examples~\ref{1vBloch1} and~\ref{1vBloch2}); in fact, any connected $1$-vertex graph is dense). By Theorem~\ref{thm:homothetic}, $D_Q$ is only homothetically reducible. To show $D_Q$ is irreducible for all potentials, it suffices to show that $D_Q|_F$ is irreducible for some face $F$.

In dimensions $2$ and $3$, we have already proven that each $D|_F$ is irreducible when $F$ is not a vertex. By ~\cite[Theorem~4.2]{FS}, for a generically labeled dense $\ZZ^2$- or $\ZZ^3$-periodic graph, the zero-set of $D|_F$ is smooth (and $D|_F$ is square-free) for any nonbase face $F$. In particular, this implies that for every facet $F$, $D|_F$ is irreducible (see~\cite{GKZ}). It follows that $D|_F$ is irreducible for any nonbase face $F$.  As the nonbase facial polynomials are potential-independent, it follows that $D_Q$ is irreducible for all potentials for infinitely many choices of $Q$ (as in Example~\ref{1vBloch2}).
\begin{figure}[ht]
    \centering
      \centering
     \begin{subfigure}{0.35\textwidth}
         \centering
    \includegraphics[scale=1]{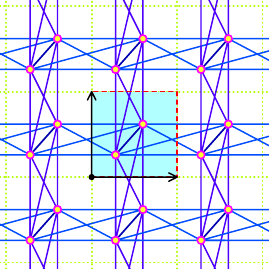}
     \end{subfigure}
     \hspace{1in}
      \begin{subfigure}{0.35\textwidth}
    \includegraphics[scale=1.15]{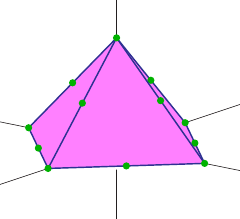}
     \end{subfigure}
    \caption{A $\ZZ^2$-periodic dense graph and corresponding Newton polytope.}
    \label{fig:densegraphpolytope}
\end{figure}

Consider the $\ZZ^2$-periodic dense graph from~\cite{DKS} shown in Figure~\ref{fig:densegraphpolytope}.
The Floquet matrix of the discrete periodic operator has entries:
\[
\begin{split}
L_{1,1}(z,\lambda) & = \alpha +  \beta_1 (2-z_1 - z_1^{-1}) + \beta_2 +\beta_3 + \gamma_1 (2 - z_2 - z_2^{-1}) + \gamma_2 + \gamma_3 + V_1 - \lambda \\
L_{1,2}(z,\lambda) & = -\alpha - \beta_2 z_1 - \beta_3 z_1^{-1} - \gamma_2 z_2 - \gamma_3 z_2^{-1}\\
L_{2,1}(z,\lambda) & = -\alpha - \beta_2 z_1^{-1} - \beta_3 z_1 - \gamma_2 z_2^{-1} - \gamma_3 z_2 \\
L_{2,2}(z,\lambda) & = \alpha +  \beta_4 (2-z_1 - z_1^{-1}) + \beta_2 +\beta_3 + \gamma_4 (2 - z_2 - z_2^{-1}) + \gamma_2 + \gamma_3 + V_2 - \lambda.\\
\end{split}
\] 
Where $\alpha, \beta_i, \gamma_j$ are edge labels. Let $F$ be the facet of $\newt{D}$ exposed by $(-1,-1,-1)$, then $D|_F$ has exactly the monomials terms $\lambda^2,z_1^2,z_2^2,\lambda z_1, \lambda z_2$, and $z_1z_2$.  By~\cite[Theorem~4.2]{FS}, $D|_F$ is irreducible. As $\lambda z_1$, $\lambda z_2$, and $\lambda^2$ are terms of $D|_F$, $\Span_{1,\sigma}(D|_F)$ and $\Span_{2,\sigma}(D|_F)$ both equal $1$ for $\sigma = \{1,2\}$. Thus by Corollary~\ref{cor:coprime}, $D_Q|_F$ is irreducible for any choice of $q_1$ and $q_2$. By Corollary~\ref{cor:irred}, $D_Q$ is irreducible for all potentials.
\hfill$\diamond$\medskip
\end{Example}

\subsection{Fermi Varieties}~\label{subsec:Fermi}
The $d$-dimensional \demph{cross-polytope}, $\demph{CP_d}$, is the polytope with vertices given by the $2d$ $d$-dimensional vectors that are either $1$ or $-1$ in one coordinate, and $0$ elsewhere. Letting $A = (a_1,\dots, a_d) \in \NN^d$, we define $\psi_A \vcentcolon \RR^d \to \RR^d$ to be the linear map given by $(v_1, v_2, \dots, v_d) \mapsto (a_1 v_1, a_2 v_2, \dots, a_d v_d)$. A $d$-dimensional \demph{$A$-dilated cross-polytope} is the image $\psi_A(CP_d)$. The image of a polytope $P$ under the map $\psi_A$ is called the \demph{A-dilation} of $P$. 

By Remarks~\ref{RM:HOMIRED} and ~\ref{rm:alpha2}, we may use our methods to discuss the irreducibility of Fermi varieties for $\ZZ^d$-periodic graphs when $d >2$. We note that when $d=2$, the Newton polytope of the polynomial defining a Fermi variety is $2$-dimensional, and thus we cannot apply the theory of only homothetic reducible polynomials that was developed in Section~\ref{Sec:OHD2OHR} (as the faces in a strong chain must share at least an edge). As with Bloch varieties, we begin with $1$-vertex graphs.
\medskip
\begin{Example}~\label{1vFermi}
Let $d$ be some integer greater than $2$. Fix $\lambda_0 \in \CC$. Consider a $1$-vertex $\ZZ^d$-periodic graph with fundamental domain $W$, such that the convex hull of $\calA(W)$ is an $A$-dilated cross-polytope for some $A=(a_{1},\dots,a_{d}) \in \NN^d$ such that $\gcd(a_1,\dots, a_d) = 1$. For example, the $3$-dimensional square lattice is a $1$-vertex $\ZZ^3$-periodic graph, and the convex hull of its support, $\calA(W)$, is the cross-polytope. Another example (although $d < 3$) is the left-hand graph of Figure~\ref{fig:grid}, the convex hull of its support, $\calA(W)$, is the $(3,2)$-dilated cross-polytope.

As $1$-vertex graphs are dense periodic graphs, by \cite[Lemma 4.3]{FS}, we have that $\newt{D(z,\lambda_0)} = \text{conv}(\calA(W))$. It follows that any facet of $\newt{D(z,\lambda_0)}$ is a pyramid with apex $(a_1,0, \dots, 0)$ or $(-a_1,0 \dots, 0)$ and thus is only homothetically decomposable. As $d>2$, we have that any two vertices can be connected by a strong chain of only homothetically decomposable facets, and thus $\newt{D(z,\lambda_0)}$ is only homothetically decomposable. By arguments of \cite{Gao}, as $\gcd(a_1,\dots, a_d) = 1$, if $F$ is a facet then $D|_F(z,\lambda_0)$ is irreducible. Let $Q \in \NN^{d}$ be such that $\gcd(q_i,a_i)=1$ for each $i$ and $\gcd(q_1,\dots,q_d)=1$. By the arguments given in Section~\ref{subsec:BlochV} on $1$-vertex graphs, the facial polynomials of $D_Q(z,\lambda_0)$ are independent of the potential and of $\lambda_0$.

Let $F$ be a facet of $\newt{D_Q(z,\lambda_0)}$. As $\gcd(q_i,a_i)=1$ for each $i$, $D_{\sigma \odot Q}|_F(z,\lambda_0)$ is irreducible for each $\sigma \in \binom{[d]}{d-1}$  by Lemma~\ref{lem:coprime} (here, a power of $z_i$, for $i \in \bar{\sigma}$, acts as the constant mentioned in Remark~\ref{rm:alpha2}). Since $\gcd(q_1,\dots,q_d)=1$, $D_Q|_F(z,\lambda_0)$ is irreducible by Theorem~\ref{TH:1}. By Corollary~\ref{cor:irred}, $D_Q(z,\lambda_0)$ is irreducible for all potentials.
\hfill$\diamond$\medskip
\end{Example}

\begin{Example}~\label{DFermi}
Let $d$ be an integer greater than $2$ and let $\lambda_0\in \CC$. Consider a $\ZZ^d$-periodic dense graph $\Gamma$ with generic edge labels. As in Example~\ref{1vFermi}, each facial polynomial of $D_Q(z,\lambda_0)$ is independent of the potential and of $\lambda_0$. For a given edge label, if one can show all facial polynomials of $D(z,\lambda_0)$ are only homothetically reducible, then all facial polynomials of $D_Q(z,\lambda_0)$ are only homothetically reducible for any $Q \in \NN^d$. If there exists an irreducible facial polynomial $D|_F(z,\lambda_0)$, one can use Section~\ref{sec:2-structure} to find $Q \in \NN^d$ where $D_Q|_F(z,\lambda_0)$ is irreducible. Finally, we can apply Theorem~\ref{thm:homothetic} to conclude irreducibility of $D_Q(z,\lambda_0)$. As in Example~\ref{ex:DenseBloch}, by~\cite[Theorem~4.2]{FS} when $d=3$, $D|_F(z,\lambda_0)$ is irreducible. It follows that there are infinitely many $Q$ such that $D_Q|_F(z,\lambda_0)$ is irreducible.

Consider the 3-dimensional dense graph shown in Figure~\ref{fig:3ddensepolytope}. This is the $3$-dimensional analog of the dense graph from Example~\ref{ex:DenseBloch}. The associated discrete periodic operator $L$ has a Floquet matrix with entries:
\[
\begin{split} 
L_{1,1}(z,\lambda) & = \alpha +  \beta_1 (2-z_1 - z_1^{-1}) + \beta_2 +\beta_3 + \gamma_1 (2 - z_2 - z_2^{-1})\\
& \ \ \ \ \  + \gamma_2 + \gamma_3 + \epsilon_1(2 - z_3 - z_3^{-1}) +\epsilon_2 + \epsilon_3 + V_1 - \lambda \\
L_{1,2}(z,\lambda) & = -\alpha - \beta_2 z_1 - \beta_3 z_1^{-1} - \gamma_2 z_2 - \gamma_3 z_2^{-1} - \epsilon_2 z_3 - \epsilon_3 z_3^{-1} \\
\end{split}\]
\[
\begin{split} 
L_{2,1}(z,\lambda) & = -\alpha - \beta_2 z_1^{-1} - \beta_3 z_1 - \gamma_2 z_2^{-1} - \gamma_3 z_2 - \epsilon_2 z_3^{-1} - \epsilon_3 z_3  \\
L_{2,2}(z,\lambda) & = \alpha +  \beta_4 (2-z_1 - z_1^{-1}) + \beta_2 +\beta_3 + \gamma_4 (2 - z_2 - z_2^{-1})\\ 
& \ \ \ \ \ + \gamma_2 + \gamma_3 +\epsilon_4(2 - z_3 - z_3^{-1}) +\epsilon_2 + \epsilon_3 + V_2 - \lambda.\\
\end{split}
\] 
Here, $\alpha, \beta_i, \gamma_j$, and $\epsilon_k$ are edge labels. 
\begin{figure}[ht]
    \centering
         \hspace{0.4in}
         \begin{subfigure}{0.4\textwidth}
    \includegraphics[scale=0.17]{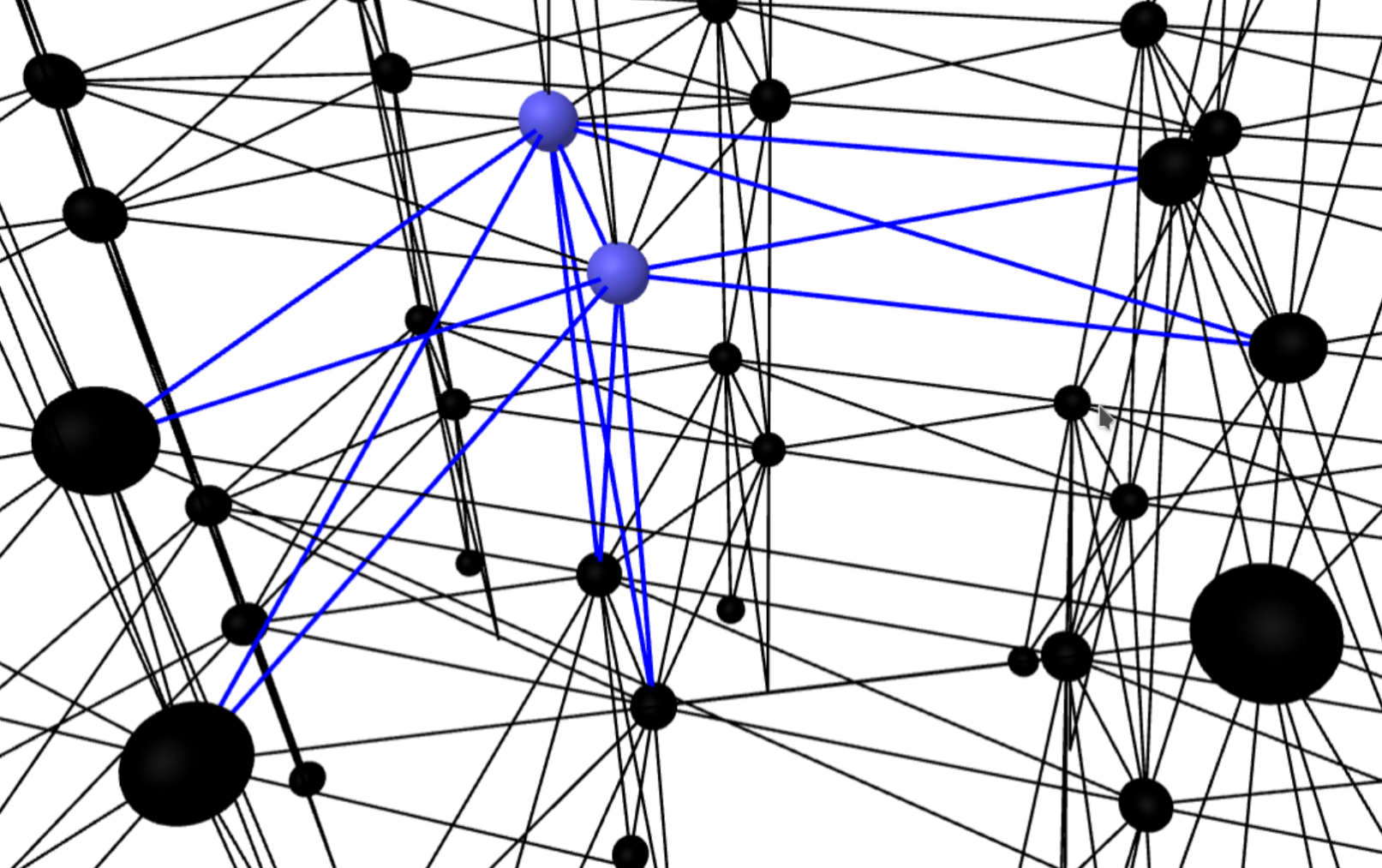}
     \end{subfigure}
     \hspace{1.3in}
      \begin{subfigure}{0.3\textwidth}
    \includegraphics[scale=1.1]{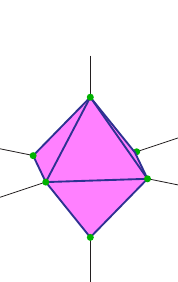}
     \end{subfigure}      
    \caption{A $3$-dimensional dense graph (left). The polytope $\newt{D(z,\lambda_0)}$ is a $3$-dimensional dilated cross-polytope (right).}
    \label{fig:3ddensepolytope}
\end{figure}

As shown in Figure~\ref{fig:3ddensepolytope}, $\newt{D(z,\lambda_0)}$ is a $(2,2,2)$-dilated cross-polytope, hence its facial polynomials are only homothetically reducible. Consider the face $F$ of $\newt{D(z,\lambda_0)}$ exposed by $w = (-1,-1,-1)$. By ~\cite[Theorem~4.2]{FS}, $D|_F(z,\lambda_0)$ is irreducible. By Lemma~\ref{lem:coprime}, $(D_{\sigma \odot Q})|_F(z,\lambda_0)$ is irreducible for each $\sigma \in \binom{[3]}{2}$ for all $Q \in \NN^d$ because $1 = \Span_{i,\sigma}(D|_F(z,\lambda_0))$ for each $i \in \sigma \in \binom{[3]}{2}$. Therefore, if $Q$ is chosen so that $\gcd(q_1,q_2,q_3) = 1$, then $D_Q|_F(z,\lambda_0)$ is irreducible by Theorem~\ref{TH:1}. By Corollary~\ref{cor:irred}, $D_Q(z,\lambda_0)$ is irreducible for all potentials. 
\hfill$\diamond$
\end{Example}

\section*{Acknowledgements}
We are grateful to Fillman, Kuchment, Liu, Matos, Shipman, and Sottile for their valuable feedback during the preparation and editing of this work. We also thank Sottile for providing some of the more visually appealing graphics. 
\def\cprime{$'$}

\bibliographystyle{amsplain}
\bibliography{final_draft}

\end{document}